\documentclass{amsart}
\usepackage{amsmath}
\usepackage{amsfonts}
\usepackage{amssymb}
\usepackage{graphicx}
\usepackage[all]{xy}
\usepackage{pdfsync}
\usepackage{mathrsfs}
\usepackage[colorlinks=true]{hyperref}

\newtheorem{theorem}{Theorem}[section]

\newtheorem{corol}[theorem]{Corollary}

\theoremstyle{definition}
\newtheorem{definition}[theorem]{Definition}
\newtheorem{example}[theorem]{Example}

\newtheorem{lemma}[theorem]{Lemma}

\newtheorem{problem}[theorem]{Problem}
\newtheorem{prop}[theorem]{Proposition}
\newtheorem{remark}[theorem]{Remark}

\numberwithin{equation}{section}

\newcommand{\Rr}{\mathbb R}

\newcommand{\set}[1]{\left\{#1\right\}}

\newcommand{\inner}[1]{\left\langle#1\right\rangle}      

\newcommand{\X}{\ensuremath{\mathfrak{X}}}
\newcommand{\F}{\ensuremath{\mathcal{F}}}
\renewcommand{\d}{\mathrm d}

\newcommand{\R}{\ensuremath{\mathcal{R}}}

\DeclareMathOperator{\GL}{GL}
\DeclareMathOperator{\Sp}{Sp}
\DeclareMathOperator{\tr}{trace}
\DeclareMathOperator{\Inv}{Inv}
\DeclareMathOperator{\Diff}{Diff}
\DeclareMathOperator{\modular}{mod}     

\newcommand{\Jet}{\text{\rm J}}
\newcommand{\jet}{\text{\rm j}}

\newcommand{\G}{\mathcal{G}}            
\newcommand{\s}{\mathbf{s}}             
\renewcommand{\t}{\mathbf{t}}           

\newcommand{\al}{\alpha}                
\newcommand{\Lie}{\mathscr{L}}          
\renewcommand{\gg}{\mathfrak{g}}        
\newcommand{\hh}{\mathfrak{h}}          
\newcommand{\ggl}{\mathfrak{gl}}          
\newcommand{\ssp}{\mathfrak{sp}}          
\newcommand{\ad}{\text{\rm ad}\,}       
\newcommand{\tto}{\rightrightarrows}    
\newcommand{\wmc}{\omega_{\textrm{MC}}}   

\newcommand{\rank}{\text{\rm rk}\,}   
\newcommand{\B}{\mathrm{F}}                    
\newcommand{\K}{\mathcal{K}}                     

\begin{document}
\title[Classifying Lie Algebroid of a Geometric Structure]{The classifying Lie algebroid of a geometric structure I: classes of coframes}
\author{Rui Loja Fernandes}
\address{Departamento de Matem\'atica, Instituto Superior T\'ecnico, 1049-001, Lisbon, Portugal}
\curraddr{Department of Mathematics, The University of Illinois at Urbana-Champaign, 1409 W. Green
Street, Urbana, IL 61801, USA}
\email{ruiloja@illinois.edu}

\author{Ivan Struchiner}
\address{Departamento de Matem\'atica, Universidade de S\~ao Paulo, Rua do Mat\~ao 1010, 
S\~ao Paulo -- SP, Brasil, CEP: 05508-090}
\email{ivanstru@ime.usp.br}

\thanks{RLF was partially supported by FCT through the Program POCI 2010/FEDER and by projects PTDC/MAT/098936/2008 and PTDC/MAT/117762/2010. IS was partially supported by FAPESP 03/13114-2, CAPES BEX3035/05-0 and by NWO}
\date{}

\begin{abstract} We present a systematic study of symmetries, invariants and moduli spaces of classes of coframes. We introduce a classifying Lie algebroid to give a complete description of the solution to Cartan's realization problem that applies to both the local and the global versions of this problem.
\end{abstract}

\maketitle
\tableofcontents

\section{Introduction}

This is the first of two papers dedicated to a systematic study of symmetries, invariants and moduli spaces of geometric structures of \emph{finite type}. Roughly speaking, a geometric structure of finite type can be described as any geometric structure on a smooth manifold which determines (and is determined by) a coframe on a (possibly different) manifold. Here by a \textbf{coframe} on an $n$-dimensional manifold $M$ we mean a generating set $\set{\theta^1,\ldots, \theta^n}$ of everywhere linearly independent $1$-forms on $M$. 

Examples of finite type geometric structures include finite type $G$-structures (see, e.g., \cite{Sternberg} or \cite{FernandesStruchiner}), Cartan geometries (see, e.g., \cite{Sharpe} or \cite{FernandesStruchiner}) or linear connections which preserve some tensor on a manifold (see, e.g., \cite{Kobayashi}).

The main problem concerning finite type structures is that of determining when two such geometric structures are isomorphic. This is a classical problem, that goes back to \'Elie Cartan \cite{Cartan}, and is usually refered to as the \emph{Cartan equivalence problem}.  When two geometric structures of the same kind have been properly characterized by coframes $\set{\theta^i}$ and $\set{\bar{\theta^i}}$ on $n$-dimensional manifolds $M$ and $\bar{M}$, the equivalence problem takes the form:

\begin{problem}[Equivalence Problem]
\label{equiv:probl}
Does there exist a diffeomorphism $\phi:M\rightarrow\bar{M}$ satisfying
\[
\phi^{\ast}\bar{\theta}^{i}=\theta^{i}
\]
for all $1 \leq i \leq n$?
\end{problem}

We will show that to a coframe there is associated a \textbf{classifying Lie algebroid} which encodes the corresponding equivalence problem. This classifying algebroid plays a crucial role in various geometric, local and global, classification problems associated with coframes.

In order to explain how this algebroid arises, recall that the solution to the equivalence problem is based on the simple fact that exterior differentiation and pullbacks commute: taking exterior derivatives and using the fact that $\{ \theta^{i} \}$ is a coframe, we may write the
\textbf{structure equations}
\begin{equation}\label{structureeqcoframe}
\d \theta^{k}= \sum_{i<j}C_{ij}^{k}(x)  \theta^{i}\wedge\theta^{j}
\end{equation}
for uniquely defined functions $C_{ij}^{k} \in C^{\infty}( M )$ known as \textbf{structure functions}. Analogously, we may write
\[\d \bar{\theta}^{k} = \sum_{i<j}\bar{C}_{ij}^{k}(\bar{x})
\bar{\theta}^{i} \wedge \bar{\theta}^{j}\text{.}\]
It then follows from $\phi^{\ast} \d \bar{\theta}^{k} = \d \phi^{\ast}\bar{\theta}^{k}$, that a necessary condition for equivalence is that
\[\bar{C}_{ij}^{k} (\phi(x)) = C_{ij}^{k}(x).\]
Thus, the structure functions are \emph{invariants} of the coframe. One can obtain more invariants of the coframe by further differentiation and this is the key remark to solve the equivalence problem.

In the opposite direction, one may consider the problem of determining when a given set of functions $C_{ij}^k$ can be  realized as the structure functions of some coframe. This problem was proposed and solved by Cartan in \cite{Cartan}. Following Bryant (see the Appendix in \cite{Bryant}), we state it in the following precise form:

\begin{problem}[Cartan's Realization Problem]
\label{realiz:probl}
One is given:
\begin{itemize}
\item an integer $n\in\mathbb{N}$,
\item a $d$-dimensional manifold $X$,
\item a set of functions $C_{ij}^{k}\in C^{\infty}(X)$ with indexes $1\leq i,j,k\leq n$,
\item and $n$ vector fields $F_{i}\in\X(X)$
\end{itemize}
and asks for the existence of
\begin{itemize}
\item an $n$-dimensional manifold $M$;
\item a coframe $\theta=\{\theta^{i}\}$ on $M$;
\item a map $h : M \rightarrow X$
\end{itemize}
satisfying
\begin{align}
\d \theta^{k} & = \sum_{i<j}(C_{ij}^{k}\circ h) \theta^{i}\wedge\theta^{j} \label{eq: dtheta}\\
\d h & = \sum_{i}(F_{i}\circ h) \theta^{i}\text{.} \label{eq: dh}
\end{align}
\end{problem}

Notice that the right hand side of equation \eqref{eq: dh} defines a bundle map $TM \to TX$ by first applying the $\theta^i$'s to the tangent vector and evaluating the vector fields $F_i$'s, and then taking the corresponding linear combination.

Most often, in concrete examples, one is given the structure functions and looks for all the possible coframes realizing it. In fact, one is interested in the following two problems:

\begin{description}
\item[Local Classification Problem] What are all germs of coframes which solve Cartan's realization problem?
\item[Local Equivalence Problem] When are two such germs of coframes equivalent?
\end{description}

Later, we shall give complete solutions to the existence problem, as well as both the classification and the equivalence problems. Before we proceed, let us present a simple, classic, but elucidating example showing that Lie theory enters into the picture.

\begin{example}[The case $d=0$]
\label{ex:constant}
Suppose that we are interested in the equivalence and classification problem for coframes whose structure functions are constant, or in the precise formulation of Cartan's realization problem given above, the case where the manifold $X$ reduces to a point (in particular, the vector fields $F_i$ vanish identically).

Necessary conditions for the existence of a solution to the realization problem are obtained from
\[\d^{2}\theta^{m}=0 \qquad (m = 1 , \ldots n).\]
These conditions imply that the the structure constants must satisfy:
\[ C_{im}^lC_{jk}^m+C_{jm}^lC_{ki}^m+C_{km}^lC_{ij}^m=0, \qquad (i,j,k,m = 1 , \ldots n),\]
i.e., they are the structure constants of a Lie algebra $\gg$ relative to some basis $\set{e_1, \ldots, e_n}$:
\[ [e_i,e_j]=C_{ij}^ke_k.\]

This condition is also sufficient. In fact, if we let $M:=G$ be any Lie group which has $\gg$ as its Lie algebra, then the components $\theta^i$ of its $\gg$-valued right invariant Maurer-Cartan form $\wmc$ relative to some basis $\set{e_1, \ldots, e_n}$ of $\gg$, have $C^k_{ij}$ as its structure functions.

Now let $(M, \theta^i, h)$ be any other solution of the realization problem. If we define a $\mathfrak{g}$-valued $1$-form on $M$ by
\[ \theta=\sum_{i}\theta^{i}e_{i},\]
then the structure equation \eqref{eq: dtheta} for the
coframe $\{ \theta^{i}\}$ implies that $\theta$ satisfies the
\emph{Maurer-Cartan equation}
\[\d \theta + \frac{1}{2}[\theta,\theta] = 0\text{.}\]
It is then well known that there exists a locally defined diffeomorphism $\psi:M\rightarrow G$ onto a
neighborhood of the identity such that
\[ \theta=\psi^{\ast}\wmc.\] 
This is sometimes refered to as \emph{the universal property of the Maurer-Cartan form}.
Thus, at least locally, there is only one solution to the realization problem, up to equivalence.
\end{example}

In general, when the structure functions are not constant, they cannot determine a Lie algebra. Instead, we have the following result which is the basic proposition which underlies all our approach:

\begin{theorem}
Solutions to Cartan's problem exist if and only if the initial data to the problem determines a Lie algebroid structure on the trivial vector bundle $A\to X$ with fiber $\Rr^n$, Lie bracket $[~,~]_A:\Gamma(A)\times\Gamma(A)\to \Gamma(A)$ relative to the standard basis of sections $\{e_1\dots,e_n\}$ given by:
\[ [e_i,e_j]_A=C_{ij}^k(x)e_k,\]
and anchor $\sharp:A\to TM$ defined by $\sharp(e_i)=F_i$.
\end{theorem}

We will call $A\to X$ the \textbf{classifying Lie algebroid} of the Cartan's realization problem. Also, the map $h:M\to X$, which determines the particular coframe we are interested in, will be called the \textbf{classifying map} of the realization. 

We will see in Section \ref{subsec:Algbrd}, that given a coframe $\theta$ on a manifold $M$, there is associated to it a Cartan Realization Problem, and hence a classifying Lie algebroid $A_\theta$. This will be called the \textbf{classifying Lie algebroid of the coframe}. It is important to make the following distinctions:
\begin{enumerate}
\item[(i)] To a \emph{single} coframe $\theta$, there is associated a Cartan realization problem and, hence, a classifying Lie algebroid $A_\theta$.
\item[(ii)] To a Cartan realization problem, there is associated a classifying Lie algebroid $A$ and a \emph{family} of coframes (its solutions). 
\end{enumerate} 
Note that distinct Cartan realization problems can share the same coframe as solutions. However, we have the following:

\begin{prop}
Let $(n,X,C_{ij}^k,F_i)$ be a Cartan realization problem with associated classifying algebroid $A$ and let $(M,\theta,h)$ be a connected realization. Then $h(M)$ is an open subset of an orbit of $A$ so that $A|_{h(M)}\to h(M)$ is also a Lie algebroid. Moreover, there exists a Lie algebroid morphism  from $A|_{h(M)}$ to $A_{\theta}$ (the classifying algebroid of the coframe $\theta$) which is a fiberwise isomorphism that covers a submersion $h(M) \to X_{\theta}$.
\end{prop}

In general, a Lie algebroid does not need to be associated with any Lie groupoid (see \cite{CrainicFernandes} for a complete discussion of this point). However, assume for now that this is the case for the classifying Lie algebroid $A\to X$ and denote by $\G\tto X$ a Lie groupoid integrating it. One can then introduce the Maurer-Cartan form $\wmc$ on $\G$ and it follows from its universal property that:

\begin{theorem}
Let $\wmc$ be the Maurer-Cartan form on a Lie groupoid $\G$ integrating the classifying Lie algebroid $A$ of a Cartan's realization problem. Then the pull-back of $\wmc$ to each $\s$-fiber induces a coframe which gives a solution to the realization problem. Moreover, every solution to Cartan's problem is locally equivalent to one of these. 
\end{theorem}

This settles the classification and the equivalence problems. It also shows one of the main advantages of using our Lie algebroid approach: the classifying Lie algebroid (or more precisely its local Lie groupoid) gives rise to a recipe for constructing explicit solutions to the problem. The solutions are the $\s$-fibers of the groupoid equipped with the restriction of the Maurer-Cartan forms. There are several other advantages in using this Lie algebroid approach.

First of all, it turns out that properties of the classifying Lie algebroid are closely related to geometric properties of the coframes it classifies. For example, let us call any self-equivalence of a coframe which preserves the classifying map $h$ a \textbf{symmetry of the realization}. Then: 

\begin{prop}
Let $A\to X$ be the classifying Lie algebroid of a Cartan's realization problem. Then:
\begin{enumerate}
\item To every point on $X$ there corresponds a germ of a coframe which is a solution to the realization problem;
\item Two such germs of realizations are equivalent if and only if they correspond to the same point in $X$;
\item The isotropy Lie algebra of $A\to X$ at some $x\in X$ is isomorphic to the \emph{symmetry Lie algebra} of the corresponding germ of realization.
\end{enumerate}
\end{prop}

Second of all, the Lie algebroid approach is well fitted to the study of global aspects of the theory. The first of these global issues that we will consider is the

\begin{problem}[Globalization Problem]
\label{prob:globalization}
Given a Cartan's problem with initial data $(n, X, C^k_{ij}, F_i)$ and two germs of coframes $\theta_0$ and $\theta_1$ which solve the problem, does there exist a global connected solution $(M,\theta,h)$ to the realization problem for which $\theta_0$ is the germ of $\theta$ at a point $p_0 \in M$ and $\theta_1$ is the germ of $\theta$ at a point $p_1 \in M$?
\end{problem}

Another important global issue is the \emph{global equivalence problem}. In general, the classifying Lie algebroid does not distinguish between a realization and its covers. Here, by  a \textbf{realization cover} of $(M,\theta,h)$ we mean a realization of the form $(\tilde{M}, \pi^{\ast}\theta, \pi^{\ast}h)$ where $\pi: \tilde{M} \to M$ is a covering map.

It is then natural to consider the following equivalence relation on the set of realizations of a Cartan's problem. Two realization $(M_1, \theta_1, h_1)$ and $(M_2, \theta_2, h_2)$ are said to be \textbf{globally equivalent, up to covering} if they have a common realization cover $(M, \theta, h)$, i.e.,
\[\xymatrix{& (M, \theta, h) \ar[dl]_{\pi_1} \ar[dr]^{\pi_2} & \\
(M_1, \theta_1, h_1) & & (M_2, \theta_2, h_2).}\] 
This leads to:

\begin{problem}[Global Classification Problem]
\label{prob:global:equivalence}
What are all the solutions of a Cartan's realization problem up to global equivalence, up to covering?
\end{problem}

Of course, the solutions to these global problems rely, again, on understanding the classifying Lie algebroid.
We need to assume that the restriction of the classifying Lie algebroid to each orbit comes from a Lie groupoid. Note that this assumption is weaker than requiring integrability of the whole Lie algebroid $A$ (see \cite{CrainicFernandes}), so in this case we will say that $A$ is \emph{weakly integrable}.

For the globalization problem we then have the following solution:

\begin{theorem}
Let $A\to X$ be the classifying Lie algebroid of a Cartan's realization problem. Two germs of coframes which belong to the same global connected realization correspond to points in the same orbit of $A$. Moreover, if $A$ is weakly integrable then the converse is also true: two germs of coframes which correspond to points in the same orbit belong to the same global connected realization.
\end{theorem}

On the other hand, for the global equivalence problem we obtain the following result which generalizes the two main theorems concerning global equivalences due to Olver \cite{Olver}:

\begin{theorem}
Assume that the classifying Lie algebroid $A\to X$ of a Cartan's realization problem is weakly integrable.
Then any solution to Cartan's problem is globally equivalent, up to covering, to an open set of an $\s$-fiber of the source 1-connected Lie groupoid integrating $A|_{L}$, for some leaf of $A$. 
\end{theorem}

Finally, the classifying Lie algebroid can also be used to produce invariants of a geometric structure of finite type, as well as to recover classical results about their symmetry groups. In fact, if the geometric structure is characerized by a coframe $\theta$ on a manifold $M$, with associated classifying algebroid $A_\theta$, then we can view the coframe as a Lie algebroid morphism $\theta: TM \to A_{\theta}$ (which covers the classifying map $h$). As a general principle, the coframe (viewed as a Lie algebroid morphism) pulls back invariants of the classifying Lie algebroid $A_\theta$ to invariants of the coframe. To illustrate this point of view we will introduce a new invariant called the \emph{modular class of a coframe} which is obtained as the pullback of the modular class of its classifying Lie algebroid.

We can summarize this outline of the paper by saying that there are essentially two ways in which we use the existence of a classifying Lie algebroid for a Cartan's realization problem. 

\begin{itemize}
\item For a \emph{class} of geometric structures of finite type whose moduli space (of germs) is finite dimensional, we can set up a Cartan's realization problem whose classifying Lie algebroid gives us information about its local equivalence and classification problems, as well as its globalization and global classification problems.
\item For a \emph{single} geometric structure of finite type, we can set up a Cartan's realization problem whose classifying Lie algebroid describes the symmetries of the geometry, and also provides a recipe for producing invariants of the structure.
\end{itemize}

In a sequel to this paper \cite{FernandesStruchiner}, we will discuss finite type $G$-structures and show how one can associate a classifying Lie algebroid to such a structure. There we will also present several geometric examples related to torsion-free connections on $G$-structures.

Part of the results described in this paper were obtained by the second author in his PhD Thesis \cite{Struchiner}.

\section{Equivalence of coframes and realization problems}%
\label{sec: Equivalence of Coframes}         %

In this section, we will explain how coframes give rise to Cartan realization problems. First we consider the local problem, where we assume the coframe to be fully regular at some point. Then we consider the global problem for coframes which are fully regular at every point. Our exposition may be seen as a Lie algeboid version of the classical approach, which can be found in the monographs of Olver \cite{Olver} and Sternberg \cite{Sternberg}, to which we refer the reader for more details.

\subsection{Local equivalence}                                   %
\label{subsec:coframes:1}                                        %

Let $\{\theta^i\}$ and $\{\bar{\theta}^{i}\}$ be coframes on two $n$-dimensional manifolds $M$ and $\bar{M}$. Throughout this text we will use unbarred letters to denote objects on $M$ and barred letters to denote objects on $\bar{M}$. Any (locally defined) diffeomorphism $\phi:M\to \bar{M}$ such that $\phi^*\bar{\theta}^i=\theta^i$ will be called an \textbf{(local) equivalence}.

As we saw in the Introduction, the structure equations \eqref{structureeqcoframe}
\[
\d \theta^{k}=\sum_{i<j}C_{ij}^{k}(p)  \theta^{i}\wedge\theta^{j}%
\]
and the structure functions $C_{ij}^{k}\in C^{\infty}(M)$ play a crucial role in the study of the equivalence problem. For example, the structure functions of any coframe $\bar{\theta}$ equivalent to $\theta$ must satisfy
\[\bar{C}_{ij}^{k}\left(  \phi(p)  \right)  =C_{ij}^{k}(p)  \text{,}\]
so we may use the structure functions to obtain a set of necessary conditions for solving the equivalence problem. They are examples of \emph{invariant functions}:

\begin{definition}
A function $I\in C^{\infty}(M)$ is called an \textbf{invariant
function of a coframe $\theta=\{\theta^{i}\}$} if for any locally defined
self equivalence $\phi:U\to V$, where $U,V\subset M$ are open sets, one has
\[I\circ\phi=I\text{.}\]
We denote by $\Inv(\theta)\subset C^\infty(M)$ the space of invariant functions of the coframe $\theta$.
\end{definition}

Now, for any function $f\in C^{\infty}(M)$ and coframe $\theta$ one can define
the coframe derivatives $\frac{\partial f}{\partial\theta^{k}}\in C^{\infty}(M)$
as being the coefficients of the differential of $f$ when
expressed in terms of the coframe $\{\theta^{i}\}$,
\[\d f=\sum_{k}\frac{\partial f}{\partial\theta^{k}}\theta^{k}\text{.}\]
Using the fact that $\d\phi^{\ast}=\phi^{\ast}\d$, it follows
that:

\begin{lemma}
If $I\in\Inv(\theta)$ is an invariant function of a coframe $\theta=\{\theta^{i}\}$,
then so is $\frac{\partial I}{\partial\theta^{k}}$ for all $1\leq
k\leq n$. 
\end{lemma}

Obviously, if $\phi:M\to\bar{M}$ is an equivalence between coframes $\theta=\{\theta^{i}\}$ and $\bar{\theta}=\{\bar{\theta}^{i}\}$, we obtain a bijection $\phi^*:\Inv(\bar{\theta})\to\Inv(\theta)$. So, for example, we have as necessary conditions for equivalence that the structure functions and its coframe derivatives of any order 
\[ C_{ij}^{k},\frac{\partial C_{ij}^{k}}{\partial\theta^{l}}, \ldots, \frac{\partial^{s}C_{ij}^{k}}{\partial\theta^{l_{1}}\cdots\partial\theta^{l_{t}}},\ldots\] 
must correspond under the equivalence. This gives an infinite set of conditions for equivalence. However, observe that if invariants $f_{1},...,f_{l}\in\Inv(\theta)$ sastisfy a functional relationship $f_{l}=H(f_{1},...,f_{l-1})$, then the corresponding elements $\bar{f}_{1},...,\bar{f}_{l}$ in $\Inv(\bar{\theta})$ must satisfy the same functional relation
\begin{equation*}
\bar{f}_{l}=H\left(  \bar{f}_{1},...,\bar{f}_{l-1}\right)
\end{equation*}
for the same function $H$. This shows that we do not need to deal with all the
invariant functions in $\Inv(\theta)$, but only with those that are
functionally independent.

For any subset $\mathcal{C} \subset C^{\infty}(M)$ we define the \textbf{rank of $\mathcal{C}$ at $p\in M$}, denoted by $r_{p}(\mathcal{C})$, to be the dimension of the vector space spanned
by $\{\d_p f : f \in \mathcal{C}\}$. Also, we say that $\mathcal{C}$ is \textbf{regular at $p$} if $r_{p^{\prime}}(\mathcal{C}) = r_{p}(\mathcal{C})$ for all $p^{\prime}$ in a neighborhood of $p$ in $M$. Finally, we say that $\mathcal{C}$ is \textbf{regular}, if it is regular at every point $p \in M$.

\begin{definition}
A coframe $\theta$ is called \textbf{fully regular at $p\in M$} if the set $\Inv(\theta)$ is regular at $p$. It is called \textbf{fully regular} if it is fully regular at every $p \in M$.
\end{definition}

By the implicit function theorem, if $\theta$ is a fully regular coframe in $M$ of rank $d$ at $p\in M$, then we can find invariant functions $\{h_{1},...,h_{d}\}\subset\Inv(\theta)$, independent in a neighborhood of $p$, such that every $f\in\Inv(\theta)$ can be written as
\begin{equation*}
f = H(h_{1},...,h_{d})
\end{equation*}
in a neighborhood of $p$. In particular, we can assume that there is an open set $U\subset M$ containing $p$, where this independent set of invariant functions determine a smooth map:
\[
h : U \subset M \rightarrow \mathbb{R}^{d},\quad h (p) = (h_{1}(p),...,h_{d}(p)),
\]
onto an open set $X \subset \mathbb{R}^{d}$. Note that $C_{ij}^k=C_{ij}^k(h_1,\dots,h_d)$, so the structure functions are smooth functions $C_{ij}^{k} \in C^{\infty}(X)$. Also, differentiating $h_{a}$, we find
\[
\d h_{a}=\sum_{i} (F_{i}^{a}\circ h) ~\theta^{i},
\]
for some smooth functions $F_i^a\in C^{\infty}(X)$, so we obtain $n$ vector fields on $X$:
\[F_i = \sum F^a_i \frac{\partial}{\partial h_a}.\] 
Since the functions $C_{ij}^{k}$ and the vector fields $F_i$ satisfy
\begin{align*}
\d \theta^{k}  &  = \sum_{i<j}C_{ij}^{k}(h) \theta^{i} \wedge
\theta^{j}\\
\d h  &  =  \sum_{i}F_{i}(h)\theta^{i}
\end{align*}
we conclude that a coframe which is fully regular at $p$ determines the initial data of a Cartan's Realization Problem (Problem \ref{realiz:probl}). 

\begin{remark}
\label{rem:uniqueness}
The data depends on the choice of the functions $\{h_{1},...,h_{d}\}\subset\Inv(\theta)$. However, if $\{h'_{1},...,h'_{d}\}\subset\Inv(\theta)$ is a different choice, defining an open set $X'$, giving rise to data ${C'}_{ij}^{k}\in C^{\infty}(X')$ and $F'_i\in\X(X')$, then there is a diffeomorphism $\phi:X\to X'$ defined from a neighborhood of $h(p)$ to a neighborhood of $h'(p)$, which makes the diagram commutative:
\[\xymatrix{&U \ar[dl]_{h}\ar[dr]^{h'}& \\ 
X\ar[rr]_{\phi}& & X'}\]
and is such that $\phi^*{C'}_{ij}^{k}=C_{ij}^{k}$ and $\phi_*F_i=F'_i$.
\end{remark}

\subsection{Global equivalence}                                  %
\label{subsec:coframes:2}                                        %

So far, we have only considered a neighborhood of a point where the coframe is fully regular. We consider now the case where $\theta$ is a fully regular coframe at every point of $M$. 

\begin{definition}
We will say that two points $p$ and $q$ of $M$ are \textbf{locally formally equivalent} if there exist neighborhoods $U_p$ and $U_q$ of $p$ and $q$ in $M$ and a diffeomorphism $\phi: U_p \to U_q$ such that $\phi(p) = q$ and 
\begin{equation}
\label{formally equivalent}
f (\phi(p^{\prime})) = f(p^{\prime}), \quad \text{ for all } p^{\prime} \in U_p\text{ and }  f\in \Inv(\theta).
\end{equation}
\end{definition}

\begin{remark}
\label{rem:formal:local:equiv}
We note that $p$ and $q$ are locally formally equivalent if and only if $f(p)=f(q)$ for all $f\in\Inv(\theta)$. The crucial point is that since $\Inv(\theta)$ is closed under the operation of taking coframe derivatives, it follows that if $f(p) = f(q)$ for all $f \in \Inv(\theta)$, and $f_1, \ldots, f_d \in \Inv(\theta)$ are functionally independent in a neighborhood of $p$, then they are also functionally independent in some neighborhood of $q$. We shall see later, in Proposition \ref{prop: formal equivalence}, that local formal equivalence and local equivalence actually coincide.
\end{remark}

With this definition, we have:

\begin{prop}
Let $\theta$ be a fully regular coframe on $M$ of rank $d$ and denote by $\sim$ the local formal equivalence relation. The quotient:
\[X_\theta := M{/}\sim\]
has a natural structure of a smooth manifold of dimension $d$.
\end{prop}

\begin{proof}
On the the orbit space $X_\theta := M{/}\sim$ we consider the quotient topology and we denote by $\kappa_\theta:M\to X_\theta$ the projection. Also, if $U$ is an open subset of $M$, then we will denote by $\Inv_U(\theta)$ the restriction of $\Inv(\theta)$ to $U$. 

We start by choosing a cover of $M$ by a family of open sets $\set{U_{\lambda}}_{\lambda \in \Lambda}$ such that for each $\lambda \in \Lambda$ there exist  elements $h^{\lambda}_1, \ldots , h^{\lambda}_d \in \Inv(\theta)$ whose restriction to $U_{\lambda}$ are functionally independent and generate $\Inv_{U_{\lambda}}(\theta)$. Thus, on each $U_{\lambda}$ we obtain a smooth \emph{open} map $h^{\lambda}: U_{\lambda} \to \Rr^d$ defined by 
\[p \mapsto (h^{\lambda}_1(p), \ldots, h^{\lambda}_d(p)).\]
The image of this map is an open subset $h^{\lambda}(U_{\lambda}) \subset \Rr^d$, and we have an induced  map on the orbit space $\bar{h}^{\lambda}: X_{\lambda} \to \Rr^d$, where $X_{\lambda}:=\kappa_\theta(U_{\lambda})\subset X_\theta$.

\begin{lemma}
\label{lem:aux:1}
The sets $X_\lambda$ form an open cover of $X_\theta$ and each $\bar{h}^{\lambda}: X_{\lambda} \to \Rr^d$ is a homeomorphism onto the open set $h^{\lambda}(U_{\lambda})\subset \Rr^d$.
\end{lemma}

\begin{proof}[Proof of Lemma \ref{lem:aux:1}]
The set $X_\lambda=\kappa_\theta(U_\lambda)$ is open if and only if the saturation of $U_\lambda$, defined by
\[ \widetilde{U}_\lambda:=\{q\in M:q\sim p\text{ for some }p\in U_\lambda\},\]
is open in $M$. This follows because if $q\in \widetilde{U}_\lambda$ then there exists $p\in U_\lambda$ and a diffeomorphism $\phi: U_p \to U_q$, defined on neighborhoods of $p$ and $q$, such that \eqref{formally equivalent} holds. Then $\phi(U_p\cap U_\lambda)$ is an open set containing $q$, which is contained in $\widetilde{U}_\lambda$. We conclude that the saturation $\widetilde{U}_\lambda$ is open, therefore $X_\lambda$ is also open.

The induced map $\bar{h}^{\lambda}: X_{\lambda} \to \Rr^d$ is continuous (by definition of the quotient topology), open (since  $h^{\lambda}$ is open) and gives a 1:1 map onto the open set $h^{\lambda}(U_{\lambda})\subset \Rr^d$. Hence,   $\bar{h}^{\lambda}: X_{\lambda} \to h^{\lambda}(U_{\lambda})$ is a homeomorphism.
\end{proof}

Now observe that if there exists a diffeomorphism defined on an open subset $W_{\lambda}\subset U_{\lambda}$ with values in another open subset $W_{\mu} \subset U_{\mu}$,
\[\phi: W_{\lambda} \to W_{\mu},\]
and such that 
\[ f (\phi(p)) = f(p), \quad \text{ for all } p \in W_{\lambda}, \text{ and } f\in \Inv(\theta), \]
(for example if $U_{\lambda}\cap U_{\mu} \neq \emptyset$, and $W_{\lambda} = U_{\lambda} \cap U_{\mu} = W_{\mu}$) then the restriction of each $h^{\mu}_i$ to $W_{\mu}$ can be written as a smooth function of the (restriction of the)  functions $h^{\lambda}_1, \ldots h^{\lambda}_d$, i.e.,
\[h^{\mu}_i = H_{\lambda \mu}^i(h^{\lambda}_1, \ldots, h^{\lambda}_d).\]
Since both sets $\set{h^{\lambda}_i}$ and $\set{h^{\mu}_i}$ are functionally independent, we obtain a diffeomorphism
\[H_{\lambda \mu} : h^{\lambda}(W_{\lambda})\to h^{\mu}(W_{\mu}).\] 
Therefore, at the level of the orbit space, we obtain the commutative diagram:
\[ 
\xymatrix{&X_\lambda\cap X_\mu \ar[dl]_{\bar{h}^{\lambda}}\ar[dr]^{\bar{h}^{\mu}}& \\ 
\bar{h}^{\lambda}(X_\lambda\cap X_\mu)\ar[rr]_{H_{\lambda \mu}}& & \bar{h}^{\mu}(X_\lambda\cap X_\mu)}
\]
Therefore, the family $\{(X_\lambda,\bar{h}^{\lambda})\}$ gives an atlas for $X_\theta$, whose transition functions are precisely the maps $H_{\lambda \mu}$, so we conclude that $X_\theta$ is a smooth manifold, possibly non-Hausdorff.

We claim that $X_\theta$ is Hausdorff. Let $x,y\in X_\theta$ and suppose that any two neighborhoods of $x$ and $y$ intersect. To prove the claim we must show that $x=y$. For this, pick $p,q\in M$ such that $\kappa_\theta(p)=x$ and $\kappa_\theta(q)=y$. Then, by our assumption, any neighborhoods of $p$ and $q$ contain points which are in the same equivalence class. Therefore, we can choose sequences $p_n\to p$ and $q_n\to q$ such that $p_n\sim q_n$. It follows that for any $f\in \Inv(\theta)$:
\[ f(p)=\lim f(p_n)=\lim f(q_n)=f(q).\]
This implies that $p\sim q$ (see Remark \ref{rem:formal:local:equiv}) so $x=y$.
\end{proof}

We will see that the manifold $X_\theta$ plays a crucial role in the various equivalence and classification problems associated with a fully regular coframe $\theta$. Hence we suggest the following

\begin{definition}
The manifold $X_\theta$ is called the \textbf{classifying manifold} of the coframe and the natural map $\kappa_\theta: M \to X_\theta$ is called the \textbf{classifying map} of the coframe.
\end{definition}

\begin{remark}
For a arbitrary coframe $\theta$ one can still define the (singular) classifying space $X_\theta:=M/\sim$, where $\sim$ denotes local equivalence (which, now, may differ from formal local equivalence). The space $X_\theta$ is smooth on the dense open set consisting of points where the coframe is fully regular. It is an interesting problem to study the kind of singularities that may arise in this space.
\end{remark}

If $\theta=\{\theta^i\}$ is a fully regular coframe in $M$ of rank $d$, with classifying manifold $X_\theta$, by Remark \ref{rem:uniqueness}, the structure functions $C_{ij}^{k}$ can be seen as smooth functions in $X_\theta$, while the locally defined vector fields $F_i =\sum_a F_i^a\frac{\partial}{\partial h^\lambda_a}$ glue to globally defined vector fields $F_i\in\X(X_\theta)$. We conclude that:

\begin{prop}
If $\theta$ is a fully regular coframe on a manifold $M$, then it gives rise to a Cartan's realization problem with initial data $(n, X_\theta, C^k_{ij}, F_i)$ for which $(M, \theta, \kappa_\theta)$ is a solution.
\end{prop}

\section{The classifying algebroid}                       %
\label{sec:ClassAlgbrd}                                   %

In the previous section we have seen that the equivalence problem for coframes leads naturally to a Cartan realization problem. We now show that if such a problem admits solutions, then it is associated with a Lie algebroid, called the \emph{classifying algebroid} of the realization problem.

\subsection{Lie algebroids and Cartan's Problem}                 %
\label{subsec:Algbrd}                                            %

Recall that a Lie algebroid $A\to X$ is a vector bundle equipped with a Lie bracket on its space of sections $\Gamma(A)$ and a bundle map $\sharp:A\to TX$, called the anchor, such that the following Leibniz type identity holds:
\[ 
[\alpha,f\beta]=f[\alpha,\beta]+(\Lie_{\sharp(\alpha)}f)\beta,\quad (\alpha,\beta\in\Gamma(A),\ f\in C^\infty(X)).
\]
We refer to \cite{CannasWeinstein,CrainicFernandes:lecture} for details on Lie algebroids. If the vector bundle $A$ is trivial, so that we have a basis of global sections $\{\al_1,\dots,\al_n\}$, then we can express the bracket by:
\begin{equation}
\label{eq:bracket:algbrd}
[\al_i,\al_j]=\sum_{k=1}^n C_{ij}^k\al_k,
\end{equation}
for some structure functions $C_{ij}^k\in C^\infty(U)$. The anchor, in turn, is characterized by the vector fields $F_i\in\X(X)$ given by:
\begin{equation}
\label{eq:anchor:algbrd} 
F_i:=\sharp(\al_i), \quad (i=1,\dots,n).
\end{equation}
The Leibniz identity and the Jacobi identity for the bracket, lead to the following set of equations:
\begin{align}
\label{eq:struct:algbrd:1}
\left[ F_i, F_j \right] & = \sum_k C^k_{ij}F_k\\
\label{eq:struct:algbrd:2}
\Lie_{F_j}C^i_{kl} + \Lie_{F_k} C^i_{lj} + \Lie_{F_l} C^i_{jk}&  =  
\sum_m( C_{mj}^{i}C_{kl}^{m}+C_{mk}^{i}C_{lj}^{m}+C_{ml}^{i}C_{jk})
\end{align} 
Conversely, if one is given functions $C_{ij}^k\in C^\infty(X)$ and vector fields $F_i\in\X(X)$ satisfying these equations, then one obatins a Lie algebroid structure on the trivial vector bundle $A:=X\times\Rr^n$ with Lie bracket \eqref{eq:bracket:algbrd} and anchor \eqref{eq:anchor:algbrd}.

\begin{remark}
If one also chooses local coordinates $(x_1, \ldots, x_d)$ on a domain $U\subset X$, so that the vector fields $F_i$ are expressed as:
\[ F_i=\sum_a F^a_i \frac{\partial}{\partial x_a},\]
then we obtain the complete set of \emph{structure functions} $C_{ij}^k,F_i^a\in C^\infty(U)$ of the Lie algebroid $A$. Equations \eqref{eq:struct:algbrd:1} and \eqref{eq:struct:algbrd:2} become:
\begin{align*}
\sum\limits_{b=1}^{d}\left(  F_{i}^{b}\frac{\partial F_{j}^{a}}{\partial
x_{b}}-F_{j}^{b}\frac{\partial F_{i}^{a}}{\partial x_{b}}\right)
&=\sum\limits_{l=1}^{n}C_{ij}^{l}F_{l}^{a}\\
\sum\limits_{b=1}^{d}\left(  F_{j}^{b}\frac{\partial C_{kl}^{i}}{\partial
x_{b}}+F_{k}^{b}\frac{\partial C_{lj}^{i}}{\partial x_{b}}+F_{l}^{b}%
\frac{\partial C_{jk}^{i}}{\partial x_{b}}\right) &= 
\sum\limits_{m=1}
^{n}\left(  C_{mj}^{i}C_{kl}^{m}+C_{mk}^{i}C_{lj}^{m}+C_{ml}^{i}C_{jk}%
^{m}\right)
\end{align*}
We call this set of equations, or its shorter form \eqref{eq:struct:algbrd:1} and \eqref{eq:struct:algbrd:2}, the \textbf{structure equations} of the Lie algebroid.
\end{remark}

We can now give the necessary conditions for solving the realization problem. They are simply a consequence of the fact that $\d ^{2}=0$:

\begin{prop}
\label{NecessaryRealization} 
Let $(n,X,C_{ij}^{k},F_i)$ be the initial data of a Cartan's realization problem. If for every $x_0\in X$ there is a realization $(M,\theta,h)$ with $h(m_0)=x_0$, for some $m_0\in M$, then the $-C_{ij}^{k}\in C^{\infty}(X)$ and $F_i \in \X(X)$ determine the structure of a Lie algebroid $A$ over $X$.
\end{prop}

\begin{proof}
Differentiating equations \eqref{eq: dtheta} and \eqref{eq: dh} and using $\d^2=0$, we obtain:
\begin{align*}
\Lie_{F_j}C^i_{kl} + \Lie_{F_k} C^i_{lj} + \Lie_{F_l} C^i_{jk}&= - \sum_m( C_{mj}^{i}C_{kl}^{m}+C_{mk}^{i}C_{lj}^{m}+C_{ml}^{i}C_{jk})\\
\left[ F_i, F_j \right] &= - \sum_k C^k_{ij}F_k.
\end{align*}
These are just the structure equations \eqref{eq:struct:algbrd:1} and \eqref{eq:struct:algbrd:2} for the Lie algebroid defined by the trivial vector bundle $A\to X$ with fiber $\Rr^n$, with Lie bracket and anchor given by:
\begin{align*}
\sharp(\al_i)&:= F_i,\\
[\al_i, \al_j]&:= -\sum_k C^k_{ij}\al_k,
\end{align*}
where $\{\al_1, \ldots, \al_n\}$ is the canonical basis of sections of $A$.
\end{proof}

Using this proposition we introduce the following basic concept:

\begin{definition}
The Lie algebroid constructed out of the initial data of a Cartan's realization problem will be called the \textbf{classifying algebroid} of the problem.
\end{definition}

We will see later that the fact that the initial data of a Cartan realization problem determines a Lie algebroid is also sufficient for the existence of solutions, and that this Lie algebroid leads to solutions of the various classifications problems associated to coframes, which justifies its name.

\begin{remark}
\label{rem:transitive}
Notice that if we start with a fixed coframe $\theta$ on a manifold $M$, then we have an associated Cartan realization problem. The corresponding Lie algebroid will be called the \textbf{classifying algebroid of the coframe} and denoted $A_\theta\to X_\theta$. If the coframe is also the solution of another realization problem, with associated algebroid $A$, we will see later (Corollary \ref{cor:subalgbrd}) the precise relationship between $A_\theta$ and $A$. Also, note that since the vector fields $F_1,\dots, F_n$ generate at each point the tangent space to $X_\theta$, the classifying algebroid of the coframe $A_\theta$ is always transitive, while this does not need to hold with a general classifying Lie algebroid $A$ of a Cartan realization problem.
\end{remark}

\subsection{Realizations and Lie algebroid morphisms}            %
\label{subsec:morphisms}                                         %

Let $A\to X$ be the classifying Lie algebroid of a Cartan's problem. Each realization $(M,\theta,h)$ of the problem determines a bundle map:
\begin{equation}
\label{eq:alg:morphism}
\xymatrix{ 
TM\ar[r]^{\theta}\ar[d]& A\ar[d]\\
M\ar[r]_{h}& X}
\end{equation}
by setting:
\begin{align*} 
TM&\longrightarrow A=X\times\Rr^n,\\
v&\longmapsto (h(p(v)),(\theta^1(v),\dots,\theta^n(v))),
\end{align*}
where $p:TM\to M$ denotes the projection. This bundle map is a fiberwise isomorphism.

We have the following basic result:

\begin{prop}
\label{prop:morphisms}
Let $A\to X$ be the classifying Lie algebroid of a Cartan's problem. The realizations of this problem are in 1:1 correspondence with bundle maps \eqref{eq:alg:morphism} which are Lie algebroid morphisms and fiberwise isomorphisms.
\end{prop}

Before we prove this result, let us recall some basic facts about Lie algebroid morphisms (see,e.g., \cite{CrainicFernandes:lecture}). A Lie algebroid morphism is a bundle map:
\[
\xymatrix{ 
A\ar[r]^{\mathcal{H}}\ar[d]& B\ar[d]\\
X\ar[r]_{h}& Y}
\]
which commutes with anchors:
\[ \d h\circ \sharp_A=\sharp_B\circ\mathcal{H},\]
and which preserves brackets. Since we are dealing with a bundle map over different basis, to express this condition is somewhat cumbersome (see \cite{CrainicFernandes:lecture}). For us, it will be convenient to use the generalized Maurer-Cartan equation, which we now recall. 

First we introduce an arbitrary connection $\nabla$ on the vector bundle $B\rightarrow Y$. Then for \emph{any} bundle maps $\mathcal{H},\mathcal{G}:A\to B$ which commute with the anchors, we define:
\begin{enumerate}
\item[(a)] The differential operator $\d_{\nabla}\mathcal{H}$: for each $\al_1,\al_2\in\Gamma(A)$, $\d_{\nabla}\mathcal{H}(\al_{1},\al_{2})$ is the section of the pull-back bundle $\Gamma(h^{\ast}B)$ given by(\footnote{Note that, in general, $\d_{\nabla}^{2}\neq 0$ so that $\d_{\nabla}$ is not a differential.}):
\begin{equation}
\label{eq:differential}
\d_{\nabla}\mathcal{H}(\al_{1},\al_{2}):= \nabla_{\sharp_A\al_{1}}\mathcal{H}(\al_{2})
- \nabla_{\sharp_A\al_{2}}\mathcal{H}(\alpha_{1})-\mathcal{H}([\alpha_{1},\alpha_{2}])
\end{equation}
where we use the same letter $\nabla$ for the pullback connection on $h^{\ast}B$. 
\item[(b)] The bracket $[\mathcal{H},\mathcal{G}]_{\nabla}$: for each $\al_1,\al_2\in\Gamma(A)$, $[\mathcal{H},\mathcal{G}]_{\nabla}(\al_1,\al_2)$ is the section of the pull-back bundle $\Gamma(h^{\ast}B)$ given by:
\begin{equation}
\label{eq:bracket}
[\mathcal{H},\mathcal{G}]_{\nabla}(\al_1,\al_2)=-(T_{\nabla}(\mathcal{H}(\al_1),\mathcal{G}(\al_2)) +
T_{\nabla}(\mathcal{G}(\al_1),\mathcal{H}(\al_2))),
\end{equation}
where $T_{\nabla}$ is the torsion of the pullback connection on $h^*B$.
\end{enumerate}

Now we define the \textbf{generalized Maurer-Cartan equation} for a bundle map of Lie algebroids $\mathcal{H}:A\to B$, which commutes with the anchors, to be:
\begin{equation}
\label{eq:maurer:cartan}
\d_{\nabla}\mathcal{H}+\frac{1}{2}[\mathcal{H},\mathcal{H}]_{\nabla} = 0.
\end{equation}

The generalized Maurer-Cartan equation is independent of the choice of connection $\nabla$. A fundamental fact (see \cite{CrainicFernandes:lecture}) is that a bundle map of Lie algebroids $\mathcal{H}:A\to B$ which commutes with the anchors satisfies the generalized Maurer-Cartan equation if and only if it is a Lie algebroid morphism.
After these preliminaries, we can turn to the proof of our result.
\vskip 10 pt

\begin{proof}[Proof of Proposition \ref{prop:morphisms}]
First we observe that triples $(M,\theta,h)$, where $\theta$ is a coframe on $M$ and $h:M\to X$ is a smooth map are in 1:1 correspondence with bundle maps $\mathcal{H}_\theta=(h,\theta):TM\to A$ which are fiberwise isomorphisms. Then one checks easily that equation \eqref{eq: dh} is equivalent to the condition that $\mathcal{H}_\theta$ commutes with the anchors. Moreover, if $\nabla$ is the trivial connection on the (trivial) bundle $A=X\times \Rr^n$, then equation \eqref{eq: dtheta} is equivalent to the condition that $\theta$ be a solution of the Maurer-Cartan equation \eqref{eq:maurer:cartan}. This is a consequence of the following computation:
\begin{align*}
(\d_{\nabla}\theta+\frac{1}{2}[\theta,\theta]_{\nabla})(\partial_{\theta^{i}},\partial_{\theta^{j}})
& = \bar{\nabla}_{(\partial_{\theta^{i}})}\theta(\partial_{\theta^{j}}) -
\bar{\nabla}_{(\partial_{\theta^{j}})}
\theta(\partial_{\theta^{i}}) -
\theta([\partial_{\theta^{i}},\partial_{\theta^{j}}])+\\
&\hskip 50 pt -T_{\bar{\nabla}}(\theta(\partial_{ \theta^{i}}),\theta(\partial_{\theta^{j}}))\\
& = \nabla_{ \sharp \alpha_{i}} \alpha_{j} -
\nabla_{ \sharp \alpha_{j}} \alpha_{i}-\theta([\partial_{\theta^{i}},\partial_{\theta^{j}}])+\\
&\hskip 50 pt -T_{\nabla}(\alpha_{i},\alpha_{j})\\
& =[\alpha_{i},\alpha_{j}]-\sum_k\theta^k([\partial_{\theta^{i}},\partial_{\theta^{j}}])\al_k\\
& =\sum_k\left(C_{ij}^k-\d\theta^k(\partial_{\theta^{i}},\partial_{\theta^{j}})\right)\al_k\\
& =\sum_k\left(\sum_{l<m}C_{lm}^k\theta^l\wedge\theta^m(\partial_{\theta^{i}},\partial_{\theta^{j}})-\d\theta^k(\partial_{\theta^{i}},\partial_{\theta^{j}})\right)\al_k\\
\end{align*}
where $\partial_{ \theta^{i}}=\frac{\partial}{\partial \theta^{i}}$ and $\bar{\nabla}$ denotes the pullback connection on $h^{\ast}A$.
\end{proof}

\begin{corol}
Let $A\to X$ be the classifying Lie algebroid of a Cartan's problem. If $(M,\theta,h)$ is a realization of the problem then the rank of $h$ is locally constant and $h$ maps each connected component of $M$ onto an open subset of an orbit of $A$.
\end{corol}

\begin{proof}
If $\mathcal{H}_\theta=(h,\theta):TM\to A$ is the fiberwise isomorphism associated with the realization $(M,\theta,h)$ then the condition that it commutes with the anchors just means that:
\[ \d h=\sharp\circ\mathcal{H}_\theta.\]
Therefore, we conclude that for each $m_0\in M$ (i) there is a neighborhood $U$ of $m_0$ such that image $h(U_0)$ is contained in the orbit of $m_0$, and (ii) $\rank(\d_{m_0}h)$ equals the dimension of this orbit.
\end{proof}

\begin{remark}
Note that, in general, a realization of Cartan's problem may not have constant rank. What the Corollary says is that if the problem has an associated Lie algebroid then every realization has constant rank. By Proposition \ref{NecessaryRealization}, this happens if there is a solution passing through every $x_0\in X$.
\end{remark}

\section{Maurer-Cartan Forms on Lie groupoids}            %
\label{sec:Maurer-Cartan}                                 %

For a Cartan realization problem whose manifold $X$ reduces to a point, the structure functions $C_{ij}^k$ are constant and the vector fields $F_i$ vanish. In this case, the classifying algebroid is simply a Lie algebra and one can produce solutions to the realization problem by integrating this Lie algebra into a Lie group and using its Maurer-Cartan form, as was explained in Example \ref{ex:constant}. Our solution for a general Lie algebroid will follow the same pattern where we will now need to integrate the Lie algebroid to a Lie groupoid. Therefore, we will need to discuss first Maurer-Cartan forms on Lie groupoids and their universal property. 

\subsection{Definition of a Maurer-Cartan Form}                  %
\label{subsec:forms}                                             %

In this paragraph we define Maurer-Cartan forms on Lie groupoids. These turn out to be foliated differential forms with values in the Lie algebroid.

Let $\F$ be a foliation on $M$. Recall that an \textbf{$\F$-foliated $k$-form on $M$} is a section of $\wedge^k T^{\ast}\F$, i.e., it is a $k$-form on $M$ which is only defined on $k$-tuples of vector fields which are all tangent to the foliation.
  
\begin{definition}
An $\F$-foliated differential $1$-form on $M$ with values in a Lie algebroid $A\rightarrow X$
is a bundle map
\[
\xymatrix{ T\F \ar[d] \ar[r]^\theta & A \ar[d] \\
M \ar[r]_h & X }
\]
which is compatible with the anchors, i.e., such that $\sharp(\theta(\xi))=\d h(\xi)$, $\forall\xi\in T\F$.
\end{definition}

Let $\G$ be a Lie groupoid. We will use the convention that arrows in $\G$ go from right to left, so given two arrows $g$ and $h$, the product $hg$ is defined provided that $\s(h)=\t(g)$. In particular, right translation by an element $g\in \mathcal{G}$ is a map
\[
R_{g} : \mathbf{s}^{-1}(\mathbf{t}(g)) \rightarrow\mathbf{s}^{-1}(\mathbf{s}(g))
\]
so it does not make sense to say that a form on $\mathcal{G}$ is right invariant. We must restrict ourselves to $\s$-foliated differential forms. We will say that an $\mathbf{s}$-foliated differential form $\omega$ on
$\mathcal{G}$ is \textbf{right invariant} if
\[
\omega(\xi) = \omega((R_{g})_{\ast}(\xi))
\]
for every $\xi$ tangent to an $\mathbf{s}$-fiber and $g\in\mathcal{G}$. Equivalently, we will write
\[
(R_{g})^{\ast}\omega = \omega.
\]

\begin{remark}
\label{rem:invariance}
If $\omega$ is a $\s$-foliated differential form on the groupoid $\G$ with values in its Lie algebroid $A$, the notion of right invariance still makes sense provided that we assume that $\omega:T^{\s}\G\to A$ is a bundle map that covers the target map $\t:\G\to X$.
\end{remark}

On every Lie groupoid we may define a canonical, right-invariant, differential $1$-form with values in its Lie algebroid:

\begin{definition}
The \textbf{Maurer-Cartan form of a Lie groupoid $\mathcal{G}$} is the Lie algebroid
valued $\mathbf{s}$-foliated $1$-form 
\[\xymatrix{T^{\s}\G \ar[r]^{\wmc} \ar[d] & A \ar[d]\\
\G \ar[r]_{\t} & X}\]
defined by
\[
\wmc(\xi) := \d_g R_{g^{-1}}( \xi) \in A_{\t(g)}, \quad (\xi\in T_{g}^{\mathbf{s}}\mathcal{G}).
\]
\end{definition}

For any foliation $\F$ of $M$, the bundle $T\F\to M$ has an obvious Lie algebroid structure with anchor the inclusion. Hence, it makes sense to talk about the generalized Maurer-Cartan equation in this context (see Section \ref{subsec:morphisms}), and we have the following:

\begin{prop}
The Maurer-Cartan form $\wmc$ is a Lie algebroid morphism $\wmc:T^{\s}\G \to A$, hence it satisfies the generalized Maurer-Cartan equation:
\[ \d_{\nabla}\wmc+\frac{1}{2}[\wmc,\wmc]=0. \]
\end{prop}

 \begin{proof}
The definition of the Maurer-Cartan form shows that $\wmc:T^{\s}\G \to A$ preserves anchors. On the other hand, if $\xi_1$ and $\xi_2$ are the right invariant vector fields on $\G$ corresponding to sections $\al_1$ and $\al_2$ of $A$, then the definition of $\wmc$ also shows that:
 \[ \wmc([\xi_1,\xi_2])=[\al_1,\al_2], \]
so it follows that $\wmc$ also preserves brackets.
\end{proof}

\subsection{The Local Universal Property}                        %
\label{subsec:local:propert}                                    %

We will now show that any $1$-form on a differentiable manifold, with values in an integrable Lie algebroid $A$,
satisfying the generalized Maurer-Cartan equation is locally the pullback of the Maurer-Cartan form on a Lie groupoid integrating $A$. We will need the following lemma.

\begin{lemma}
Let $\F$ be a foliation on a manifold $M$ and let $\theta$ be an $\F$-foliated $1$-form (over $h$) on $M$ with
values in a Lie algebroid $A \rightarrow X$ equipped with an
arbitrary connection $\nabla$. Assume that the distribution
$\mathcal{D} = \{\ker \theta_{x} : x\in M \} \subset T\F \subset TM$ has constant rank.
Then $\mathcal{D}$ is integrable if and only if $\d_{\nabla}
\theta(\xi_{1},\xi_{2}) = 0$ whenever $\xi _{1},\xi_{2} \in \mathcal{D}$.
\end{lemma}

\begin{proof}
Choose a local basis $\xi_{1},...,\xi_{r} \in \X(\F)$ of $\mathcal{D}$ in
an open set of $M$. Then by Frobenius Theorem, $\mathcal{D}$ is
integrable if and only if $[\xi_{i},\xi_{j}] \in
\mathrm{span}\{\xi _{1},...,\xi_{r}\}$ for all $1 \leq i,j \leq
r$, which happens if and only if $\theta([\xi_{i},\xi_{j}]) = 0$
for all $1 \leq i,j \leq r$. Since $\theta(\xi_{i})= 0$ for all $1
\leq i \leq r$ we have
\begin{equation}
\d_{\nabla}\theta(\xi_{i},\xi_{j}) =
\bar{\nabla}_{\xi_{i}}\theta(\xi_{j}) -
\bar{\nabla}_{\xi_{j}}\theta(\xi_{i}) - \theta([\xi_{i},\xi_{j}])
= -\theta([\xi_{i},\xi_{j}])
\end{equation}
from which the result follows.
\end{proof}

\begin{theorem}[Local Universal Property]
\label{Universal} Let $\theta$ be a 1-form (over $h)$ on a manifold
$M$, with values in an integrable Lie algebroid $A$, that
satisfies the generalized Maurer-Cartan equation. Let $\G$ be a Lie groupoid integrating $A$ and denote by $\wmc$ its right invariant Maurer-Cartan form. Then, for each $p_0\in
M$ and $g_0\in\mathcal{G}$ such that $h(p_0) = \mathbf{t}(g_0)$, there
exists a unique locally defined diffeomorphism $\phi : M \rightarrow \mathbf{s}^{-1}(\s(g_0))$
satisfying
\[
\phi(p_0)=g_0,\quad \phi^{\ast}\wmc=\theta.
\]
\end{theorem}

\begin{remark} 
We can summarize the theorem by saying that, at least locally, there is a unique map $\phi: M \to \G$ which makes the following diagram of Lie algebroid morphisms commute:
\[\xymatrix{ TM \ar@{-->}[rr]^{\phi_{\ast}} \ar[dd]  \ar[dr]_{\theta} & & T^{\s}\G \ar[dd] \ar[dl]^{\wmc} \\
& A \ar[dd] & \\
M \ar@{-->}[rr]^{\quad \phi} \ar[dr]_h & & \G\ar[dl]^{\t} \\
& X &}\]
\end{remark}

\begin{proof}
The proof uses the usual technique of the graph, so we will construct the graph of $\phi$ by
integrating a convenient distribution. Uniqueness will then follow
from the uniqueness of integral submanifolds of a distribution. Let us set:
\[ M {_h\hskip-3 pt}\times_{\t}\mathcal{G} = \{(p,g) \in M\times\mathcal{G} : h(p) = \t(g)\}. \] 
Since $\t$ is a surjective submersion, this fibered product over $X$ is a manifold and it comes equipped with the foliation $\F$ given by the fibers of $\s \circ \pi_{\G}$. On 
$M{_h\hskip-3 pt}\times_{\t}\G$ consider the $A$-valued $\F$-foliated $1$-form
\[
\Omega=\pi_{M}^{\ast}\theta - \pi_{\mathcal{G}}^{\ast}\wmc \text{.}%
\]
Let $\mathcal{D} = \ker \Omega$ denote the associated distribution
on $M{_h\hskip-3 pt}\times_{\t}\mathcal{G}$.

In order to apply the previous lemma, we must first show that $\mathcal{D}$
has constant rank. We will do this by showing that
\[
(\d \pi_{M})_{(p,g)} \vert_{\mathcal{D}_{(p,g)}} :
\mathcal{D}_{(p,g)} \rightarrow T_pM
\]
is an isomorphism for each $(p,g) \in M
{_h\hskip-3 pt}\times_{\t}\mathcal{G}$. Note that this also implies
that if $\mathcal{D}$ is integrable then its leaf through $(p_0,g_0)$
is locally the graph of a uniquely defined diffeomorphism $\phi:M\to\s^{-1}(\s(g_0))$ satisfying $\phi(p_0)=g_0$.

Suppose that $( \d \pi_{M})_{(p,g)}(v,w) = 0$ for some
$(v,w)\in\mathcal{D}_{(p,g)}$. Since $\mathcal{D}$ is contained in $T\F$, it follows that $w$ is $\mathbf{s}$-vertical
and $\wmc(w)$ is simply the right translation of $w$ to
$\mathbf{1}_{\mathbf{t}(g)}$ and thus,
\begin{eqnarray*}
( \d \pi_{M})_{(p,g)}(v,w) = 0 & \implies & v=0 \\
& \implies & \wmc(w) = 0 \quad ( =\theta(v)) \\
& \implies & w=0 \\
& \implies & (v,w) = 0
\end{eqnarray*}
so that $( \d \pi_{M})_{(p,g)}$ is injective. Now, if $v\in
T_{p}M$ then $(v,(R_{g})_{\ast}\theta(v))$ is an element of
$\mathcal{D}_{(p,g)}$ so $(\d \pi_{M})_{(p,g)}$ is also
surjective.

Having accomplished this, we may use the preceding lemma to
complete the proof. We compute (omitting the pullbacks for
simplicity):
\begin{align*}
\d_{\nabla}\Omega &  = \d_{\nabla}\theta - \d_{\nabla}\wmc\\
& = -\frac{1}{2}[\theta,\theta] + \frac{1}{2}[\wmc,\wmc]  \text{.}
\end{align*}
Replacing $\theta=\Omega+\wmc$ we obtain
\begin{align*}
\d_{\nabla}\Omega &  = -\frac{1}{2}[\Omega + \wmc,\Omega +
\wmc] + \frac{1}{2}[\wmc,\wmc] \\
&  = -\frac{1}{2}[\Omega,\Omega] - \frac{1}{2}[\Omega,\wmc] -
\frac{1}{2}[\wmc,\Omega] \text{.}
\end{align*}
So $\d_{\nabla}\Omega(\xi_{1},\xi_{2}) = 0$ whenever
$\Omega(\xi_{1}) = 0 = \Omega(\xi_{2})$. Hence $\mathcal{D}$ is
integrable and the proof is completed.
\end{proof}

\begin{remark}
With a slight modification, the theorem above is still valid even when $A$
is not integrable. In fact, since an $A$-valued Maurer-Cartan form on $M$
\begin{equation*}
(\theta,h) \in\Omega^{1}(M,A)
\end{equation*}
is the same as a Lie algebroid morphism
\begin{equation*}
\xymatrix{ TM \ar[d] \ar[r]^\theta & A \ar[d] \\
M \ar[r]_h & X }
\end{equation*}
it follows that $h(M)$ lies in a single orbit of $A$ in $X$. By
restricting $h$ to a small enough neighborhood, we may assume that
it's image is a contractible open set $U\subset L$ in an orbit $L$
of $A$ in $X$. Then, the restriction of $A$ to $U$ is integrable
\cite{CrainicFernandes} and we may proceed as in the proof of the
theorem.
\end{remark}

As a consequence of the theorem we obtain the following useful
corollary:

\begin{corol}\label{symmetrywmc}
Let $\G$ be a Lie groupoid with Maurer-Cartan form $\wmc$ and let $x,y\in X$ be 
points in the same orbit. If $\phi : \s^{-1}(x) \rightarrow \s^{-1}(y)$ is a
symmetry of $\wmc$ (i.e., $\phi^{\ast}\wmc = \wmc$) then 
$\phi = R_g$ for some $g \in \G$.
\end{corol}

\begin{proof}
Note that the equation $\phi^{\ast}\wmc = \wmc$ only makes sense provided $\t\circ\phi=\t$ (see Remark \ref{rem:invariance}).
Therefore, $g:=\phi(\mathbf{1}_x)$ is an arrow joining $y$ to $x$ and 
\[ R_g:\s^{-1}(x) \rightarrow \s^{-1}(y), h\mapsto hg\] 
is a symmetry of $\wmc$. Since $R_g\mathbf{1}_x=g=\phi(\mathbf{1}_x)$, from the uniqueness part of Theorem \ref{Universal} we must have $\phi=R_g$.
\end{proof}

\subsection{The Global Universal Property}                       %
\label{subsec:global:propert}                                    %

There is a more conceptual proof of Theorem \ref{Universal}
that will lead us to a global version of the universal property of
Maurer-Cartan forms. 

First, observe that Theorem \ref{Universal} is a local result so we may assume 
that $M$ is simply connected. The source simply connected Lie groupoid integrating 
$TM$ is then the pair groupoid $M\times M\rightrightarrows M$. By Lie's Second Theorem
(see \cite{CrainicFernandes:lecture}), there exists a unique morphism of Lie groupoids
\[
\xymatrix{ M\times M \ar@<.5ex>[d]\ar@<-.5ex>[d]  \ar[r]^-{H} &
\mathcal{G} \ar@<.5ex>[d]\ar@<-.5ex>[d]
\\
M \ar[r]_h & X }
\]
integrating $\theta$. Now, if we are given $p_{0}\in M$ and $g_0\in\G$ with $h(p_0)=\t(g_0)$, then we may write
\[
H(p,p^{\prime}) = \phi(p)\phi(p^{\prime})^{-1}
\]
where $\phi:M\rightarrow\mathbf{s}^{-1}(\s(g_{0}))\subset\mathcal{G}$ is defined by
$\phi(p):= H(p,p_{0})g_0 \text{.}$ Note that $\phi$ satisfies
\[ \phi(p_{0})=g_0,\quad \phi^{\ast}\wmc=\theta, \]
and so it has the desired properties.

In general, when $M$ is not simply connected, the Lie algebroid morphism $\theta$ integrates to a Lie groupoid morphism
\[
\xymatrix{ \Pi_1(M) \ar@<.5ex>[d]\ar@<-.5ex>[d]  \ar[r]^-{F} &
\mathcal{G} \ar@<.5ex>[d]\ar@<-.5ex>[d]
\\
M \ar[r]_h & X }
\]
where $\Pi_{1}(M)$ denotes the fundamental groupoid of $M$. The
problem of determining when $\theta$ is globally the pullback of the
Maurer-Cartan form on a Lie groupoid $\mathcal{G}$ integrating $A$
is then reduced to determining when the morphism $F$ above factors
through the groupoid covering projection
\begin{align*}
p:\Pi_{1}(M)&\to M\times M,\\ 
p([\gamma])&:=(\gamma(1),\gamma(0)).
\end{align*}

\begin{theorem}[Global Universal Property]
\label{thm: global universal property}
Let $A$ be an integrable Lie algebroid with source simply
connected Lie groupoid $\mathcal{G}(A)$ and let
$(\theta,h)\in\Omega^{1}(M,A)$ be an $A$-valued differential
$1$-form on $M$. Then, given $p_0\in M$ and $g_0\in\G(A)$ with $h(p_0)=\t(g_0)$, there exists an everywhere defined local diffeomorphism
\[
\phi : M \rightarrow \mathbf{s}^{-1}(\s(g_{0})),\quad
\phi(p_{0})=g_0, \quad \phi^{\ast}\wmc=\theta,
\]
if and only if
\begin{description}
\item[Local obstruction] The 1-form $\theta$ satisfies the generalized Maurer-Cartan
equation and
\item[Global obstruction] The morphism $F$ integrating $\theta$ is
trivial when restricted to the fundamental group $\pi_1(M,p_{0})$ (i.e, the
isotropy group at $p_{0}$).
\end{description}
\end{theorem}

\begin{proof}
We begin by proving that both conditions are necessary. It is
clear that if $\theta=\phi^{\ast}\wmc$ then $\theta$ satisfies the
generalized Maurer-Cartan equation. So all we have to prove is
that $F$ is trivial on the isotropy group $\pi_{1}(M,p_{0})$.
Suppose that $\phi$ exists. Then the map
\[
H:M\times M\rightarrow\mathcal{G}%
\]
given by
\[
H(p,p^{\prime}) = \phi(p)\phi(p^{\prime})^{-1} \text{.}
\]
defines a Lie groupoid morphism over the map $\t \circ \phi$.

It then follows from $\phi^{\ast}\wmc = \theta$  that $\t \circ \phi = h$ and that 
$H$ integrates $\theta$. To see this, notice that if $f: M \to \s^{-1}(x)$ and $g: M \to \t^{-1}(x)$ are
smooth maps, then the differential of $\varphi(p,p^{\prime}) = f(p)\cdot
g(p^{\prime})$ is given by
\[
(\d \varphi)_{(p,p^{\prime})}(v,w) = (\d L_{f(p)})_{g(p^{\prime})}(\d g)_{p^{\prime}}(w) +
(\d R_{g(p^{\prime})})_{f(p)}({\d f})_{p}(v)
\]
for $v,w\in T_{(p,p^{\prime})}(M \times M)$. Thus, in our case we obtain, for any $v\in T_{p}M$,
\begin{align*}
(\d H)_{(p,p)}(0,v) 
&=(\d R_{\phi(p)^{-1}})_{\phi(p)}(\d \phi)_{p}(v)\\
&=\wmc(\phi_{\ast}v)\\
&=\phi^{\ast}\wmc(v)\\
&=\theta(v).
\end{align*}

Finally, since $H$ integrates $\theta$ we obtain the commutative diagram
\[
\xymatrix{
\Pi_1(M) \ar[d]_p \ar[r]^-F & \mathcal{G} \\
M \times M \ar[ru]_-H }
\]
where $p$ denotes the covering projection
$p([\gamma])=(\gamma(1),\gamma(0))$. We conclude that $F$ only depends on the
endpoints of $\gamma$ and not on it's homotopy class, proving the
claim.

Conversely, suppose both conditions are satisfied. Then by Theorem \ref{Universal}, it follows that $\theta$ is locally the pullback of $\wmc$ by a map $\phi_{\mathrm{loc}}$. However, since
$F$ only depends on the end points of paths $\gamma$ and not on their homotopy class, $F$ factors through
\[
\xymatrix{
\Pi_1(M) \ar[d]_p \ar[r]^-F & \mathcal{G} \\
M \times M \ar@{.>}[ru]_-H }
\]
Thus, by setting
\[
\phi(p):=H(p,p_{0})g_0
\]
we obtain a global map which, by the uniqueness result in Theorem \ref{Universal}, restricts to the locally defined maps $\phi_{\mathrm{loc}}$. It follows that $\phi^{\ast}\wmc = \theta$
and $\phi(p_{0})=g_0$ proving the theorem.
\end{proof}

Recall that we say that a Lie algebroid $A$ is \emph{weakly integrable} if the restriction of $A$ to any orbit is integrable. An immediate consequence of the global universal property is the following corollary:

\begin{corol}
\label{cor:contrat:mnfld}
Let $U$ be a contractible manifold, let $A\to X$ be a weakly integrable Lie algebroid and let $h:U\to X$ be a smooth map.  Any two $A$-valued 1-forms $\theta_1,\theta_2\in\Omega^1(U;A)$ over $h$ satisfying the generalized Maurer-Cartan equation are equivalent: for any $p\in U$ there exists a diffeomorphism $\phi:U\to U$, fixing $p$ and commuting with $h$, such that $\phi^*\theta_2=\theta_1$.
\end{corol}

\begin{remark}
Using the $A$-path approach to integrability, introduced in \cite{CrainicFernandes}, one can express the global obstruction condition of the preceding theorem at the infinitesimal level, i.e., without any mention to the Lie groupoid $\G$ integrating $A$. In fact, given any curve $\gamma : I \to M$, the path $\theta \circ \dot{\gamma}: I \to A$ satisfies 
\[\sharp (\theta \circ {\gamma}(t)) =  \frac{\d}{\d t}(h \circ \gamma )(t) \text{ for all } t \in I\]
and thus, is an $A$-path. We can then rewrite the condition as:
\begin{description}
\item[Global obstruction] \emph{For every loop $\gamma$ in $M$, homotopic to the constant curve at $p$, the $A$-path $\theta \circ \dot{\gamma}$ is $A$-homotopic to the constant zero $A$-path at $h(p)$.}
\end{description}
Note also that this last condition can be expressed in terms of a differential equation (see \cite{CrainicFernandes}).
\end{remark}

\section{Local Classification and symmetries}                    %
\label{sec:localclass:symm}                                      %

We now return to Cartan's realization problem. We will use the Lie groupoid integrating the classifying Lie algebroid to solve the Local Classification Problem and to study the symmetries of the problem.

\subsection{Local Classification Problem}                        %
\label{Solving the Classification Problem}                       %

The results of the previous two sections lead to a solution to the Local Classification Problem. Before we give the main result, we must say a few words about equivalence of realizations. By this we mean: 

\begin{definition}
Let $(n,X,C_{ij}^{k},F_i)$ be the initial data of a Cartan's realization. Given two realizations $(M_1,\theta_1,h_1)$ and $(M_2,\theta_2,h_2)$ a (locally defined) diffeomorphism $\phi:M_1\to M_2$ is called a (local) \textbf{equivalence of realizations} if
\[ \phi^*\theta_2=\theta_1,\quad h_2\circ \phi=h_1. \]
\end{definition}

This notion of equivalence should not be confused with the notion of equivalence of coframe.

\begin{remark}
Note that if $\theta$ is a coframe on $M$ and $A_\theta\to X_\theta$ is the corresponding classifying Lie algebroid, then we have the realization $(M,\theta,\kappa_\theta)$. In this case, a self equivalence of the realization $\phi:(M,\theta,\kappa_\theta)\to(M,\theta,\kappa_\theta)$ is the same thing as an equivalence of the coframe $\phi:(M,\theta)\to(M,\theta)$, because any self-equivalence $\phi$ of the coframe must satisfy $\kappa_\theta\circ\phi=\kappa_\theta$. Also, given two realizations $(M_1,\theta_1,h_1)$ and $(M_2,\theta_2,h_2)$, any equivalence of coframes $\phi:(M_1,\theta_1)\to(M_2,\theta_2)$ obviously determines an equivalence of realizations $\phi:(M_1,\theta_1,h_2\circ\phi)\to(M_2,\theta_2,h_2)$ but not necessarily of the original realizations, since the identity $h_2\circ \phi=h_1$ may not hold.
\end{remark}

Now we can state:

\begin{theorem}[Local Classification]
\label{thm: classification coframes}
Let $(n,X,C_{ij}^{k},F_i)$ be the initial data of a Cartan's realization problem and denote by $A\to X$ its classifying Lie algebroid. Then:
\begin{enumerate}
\item[(i)] Any realization is locally equivalent to a neighborhood of the
identity of a fiber $\s^{-1}(x_0)$ of a groupoid $\mathcal{G}$ integrating $A$, equipped with
the Maurer-Cartan form.
\item[(ii)] Two realizations are locally equivalent if and only if they correspond to the same point $x_0\in X$.
\end{enumerate}
\end{theorem}

\begin{proof}
Recall that we have a basis $\{\alpha_{1},\ldots,\alpha_{n}\}$ of global sections of the classifying Lie algebroid $A\rightarrow X$ such that:
\[ [\al_i,\al_j]= -C_{ij}^k \al_k,\quad \sharp \al_i=F_i.\]
Let us suppose for the moment that $A$ is an integrable Lie algebroid and that $\G$ is a Lie groupoid integrating $A$. We denote by $\wmc$ the Maurer-Cartan form of $\mathcal{G}$ and by $(\wmc^{1},\ldots,\wmc^{n})$ its
components with respect to the basis $\{\alpha_{1},\ldots,\alpha_{n}\}$. Then, for each $x_{0}\in X$, $(\mathbf{s}^{-1}(x_{0}),\wmc^{i},\mathbf{t})$ is a realization of Cartan's problem with initial data
$(n,X,C_{ij}^{k},F_i)$. 

A similar argument still holds when $A$ is not integrable. In this case, for each $x_{0}\in X$ we can
find a neighborhood $U\subset L$ of $x_{0}$ in the leaf $L$ containing it, such that the restriction of $A$ to $U$ is integrable to a Lie groupoid $\mathcal{G}\rightrightarrows U$. The Maurer-Cartan form of $\mathcal{G}$ takes values in $A\vert_{U}\hookrightarrow A$ so we can see it as an $A$-valued Maurer-Cartan
form. It is again clear that $(\mathbf{s}^{-1}(x_{0}),\wmc^{i},\mathbf{t})$ is a realization of
$(n,X,C_{ij}^{k},F_i)$.

If $(M,\theta^{i},h)$ is any realization of $(n,X,C_{ij}^{k},F_i)$, Proposition \ref{prop:morphisms} shows that the $A$-valued $1$-form $\theta\in\Omega^{1}(M,A)$ defined by
\[
\theta=\sum_{i}\theta^{i}(\alpha_{i} \circ h)
\]
satisfies the generalized Maurer-Cartan equation. 
\[ \d_{\nabla}\theta+\frac{1}{2}[\theta,\theta]_{\nabla}=0.\]
Therefore, if $p_{0}\in M$ is such that $h(p_{0})=x_{0}$ then, by the universal property of Maurer-Cartan forms, we can find a neighborhood $V$ of $p_{0}$ in $M$ and a unique diffeomorphism $\phi : V\rightarrow \phi (V) \subset \mathbf{s}^{-1}(x_{0})$ such that $\phi (p_{0})=\mathbf{1}_{x_{0}}$ and $\phi^{\ast}\wmc =\theta$. This means that any realization of Cartan's problem is locally equivalent to a neighborhood of the identity of an
$\mathbf{s}$-fiber of $\mathcal{G}$ equipped with the Maurer-Cartan form.

Finally, if $(M_1,\theta_1,h_1)$ and $(M_1,\theta_1,h_1)$ are two realizations of $(n,X,C_{ij}^{k},F_i)$ and $h_1(p_1)=h_2(p_2)=x_0$, then there exist open sets $V_i\subset M_i$ and (unique) diffeomorphisms $\phi_i : V_i\rightarrow \phi_i (V_i) \subset \mathbf{s}^{-1}(x_{0})$ such that $\phi_i(p_{i})=\mathbf{1}_{x_{0}}$ and $\phi_i^{\ast}\wmc =\theta_i$, $(i=1,2)$. Hence, the map $\phi_2^{-1}\circ\phi_1$ is a local equivalence from
$\phi_1^{-1}(\phi_1(V_1)\cap\phi_2(V_2))\subset M_1$ onto $\phi_2^{-1}(\phi_1(V_1)\cap\phi_2(V_2))\subset M_2$. The converse is obvious, since by definition an equivalence between two realizations $(M_1,\theta_1,h_1)$ and $(M_2,\theta_2,h_2)$ of $(n,X,C_{ij}^{k},F_i)$ must satisfy $h_2\circ \phi=h_1$.
\end{proof}

\subsection{Equivalence versus formal equivalence}               %
\label{subsec:formal:equiv}                                      %

Another consequence of the universal property of the Maurer-Cartan form is that ``equivalence'' and ``formal equivalence'' for a coframe actually coincide.

\begin{prop}
\label{prop: formal equivalence}
Let $\theta$ be a fully regular coframe on a manifold $M$. Any two points $p,q\in M$  are equivalent if and only if they are formally equivalent.
\end{prop}

\begin{proof}
Clearly, if two points are equivalent then they are formally equivalent. 
For the converse, let $\phi:U\to V$ be a formal equivalence defined on contractible open sets such that $\phi(p)=q$. Denote by $A_\theta\to X_\theta$ the classifying Lie algebroid of the coframe and by $\kappa_\theta:M\to X_\theta$ the classifying map. Then, by its very definition, we find that $\kappa_\theta(U)=\kappa_\theta(V)=W$ is an open set of $X_\theta$ and that the following diagram commutes:
\[
\xymatrix{ 
U\ar[rr]^{\phi}\ar[dr]_{\kappa_\theta}&&V\ar[dl]^{\kappa_\theta}\\
&W&}
\]
Our aim is to prove that there exists a diffeomorphism $\psi:U\to V$ such that $\psi(p)=q$ and 
$\psi^*\theta|_V=\theta|_U$.

For this, observe that since $\phi$ is a formal equivalence we have $C^k_{ij}=C^k_{ij}\circ\phi$, so 
the coframe $\phi^*(\theta|_V)$ satisfies the same structure equations as $\theta|_U$:
\[ \d(\phi^*\theta^k)=\sum_{i<j}C^k_{ij}(\phi^*\theta^i)\wedge (\phi^*\theta^j). \]
It follows that we have two Lie algebroid morphisms:
\[
\xymatrix{ 
TU\ar[r]^{\phi^*\theta}\ar[d]&{A_{\theta}}|_W\ar[d]\\
U\ar[r]_{\kappa_\theta}&W
}
\qquad
\xymatrix{ 
TU\ar[r]^{\theta}\ar[d]&{A_\theta}|_W\ar[d]\\
U\ar[r]_{\kappa_\theta}&W
}
\]
We conclude that $\phi^*(\theta|_V)$ and $\theta|_U$ are both 1-forms in $\Omega^1(U,A_\theta)$  over $\kappa_\theta$ satisfying the Maurer-Cartan equation. By Corollary \ref{cor:contrat:mnfld}, we conclude that there exists a diffeomorphism $\varphi:U\to U$,  fixing $p\in U$ and commuting with $\kappa_\theta$, such that $\varphi^*(\phi^*(\theta|_V))=\theta|_U$. The map $\psi:=\phi\circ\varphi:U\to V$ is the desired equivalence mapping $p$ to $q$.
\end{proof}

\subsection{Symmetries of Realizations}                          %
\label{subsec:symmetries}                                        %

We can also use our solution to the Local Equivalence Problem to recover a few classical results about symmetries of coframes. Many of these results can be traced back to Cartan \cite{Cartan}. The formulation presented here is based on \cite{Bryant}. Note that our only purpose is to show how the classifying Lie algebroid can be used to give very simple proofs of these facts.

\begin{definition}
Let $\theta$ be a coframe on $M$. A \textbf{symmetry} of $(M,\theta)$ is a self-equivalence, i.e., a diffeomorphism $\phi: M \to M$ such that $\phi^{\ast}\theta=\theta$. An \textbf{infinitesimal symmetry} is a vector field $\xi \in \X(M)$ such that $\Lie_{\xi}\theta=0$.
\end{definition}

Clearly, the set of infinitesimal symmetries of a coframe $(M,\theta)$ form a Lie subalgebra $\X(M,\theta)\subset\X(M)$. On the other hand, we recall that the group of diffeomorphism $\Diff(M)$ is a Fr\'echet Lie group for the compact-open $C^\infty$-topology. The set of all symmetries of a coframe $(M,\theta)$ form a topological subgroup $\Diff(M,\theta)\subset \Diff(M)$ for the induced topology. 

We also have the related notions of symmetry and infinitesimal symmetry of a realization:

\begin{definition}
If $(M,\theta,h)$ is a realization of Cartan's problem we call a diffeomorphism $\phi:M\to M$ a \textbf{symmetry of the realization} if $\phi^{\ast}\theta=\theta$ \emph{and} $h\circ\phi=h$. Similarly, one calls a vector field $\xi \in \X(M)$ an \textbf{infinitesimal symmetry of the realization} if $\Lie_{\xi}\theta=0$ \emph{and} $\d h\cdot \xi=0$. 
\end{definition}

Given a realization $(M, \theta, h)$ we will denote by $\Diff(M,\theta,h)$  the group of symmetries of the realization and by $\X(M,\theta,h)$ the Lie algebra of infinitesimal symmetries of the realization. Obviously, $\Diff(M,\theta,h)$ is a subgroup of the group $\Diff(M,\theta)$ of symmetries of the underlying coframe, while $\X(M,\theta,h)$ is a Lie subalgebra of the Lie algebra $\X(M,\theta)$ of infinitesimal symmetries of the underlying coframe.

Our next result describes the relationship between the symmetries and the classifying Lie algebroid:

\begin{prop}[Theorem A.2 of \cite{Bryant}] 
\label{prop:inf:sym:coframe}
Let $(M,\theta,h)$ be a realization of a Cartan's problem with classifying Lie algebroid $A \to X$. Then the set $\X(M,\theta,h)_p$ of germs of infinitesimal symmetries of $(M,\theta,h)$ at a point $p$ is a Lie algebra isomorphic to $\gg_{h(p)}$, the isotropy Lie algebra of $A$ at $h(p)$. In particular, if $M$ is connected then $h(M)$ is an open subset of a leaf $L$ of $A$ and
\[  \dim \X(M,\theta,h)_p = \dim M - \dim L\]
\end{prop}

\begin{proof}
Since the result is local, we may assume without loss of generality that the classifying Lie algebroid $A \to X$ is integrable. If this is not the case, we can restrict $A$ to a contractible open set $U$ in the leaf containing $h(p)$ so that $A|_U$ is integrable. Let $\G \tto X$ be a Lie groupoid with Lie algebroid $A$. Then, by the local universal property of the Maurer-Cartan form on $\G$ (Theorem \ref{Universal}), there is a neighborhood $U$ of $p$ in $M$ and a diffeomorphism $\phi: U \to \phi(U) \subset \s^{-1}(h(p))$ such that $\phi(p) = \mathbf{1}_{h(p)}$ and $\phi^{\ast}\wmc = \theta$.

It follows that $\X(M,\theta,h)_p$ coincides with the set of germs at $\mathbf{1}_{h(p)}$ of vector fields tangent to $\s^{-1}(h(p))$ which are infinitesimal symmetries of the Maurer-Cartan form. Thus, from (the infinitesimal version of) Corollary \ref{symmetrywmc} we may conclude that these may be identified with the isotropy Lie algebra of $A$ at $p$.
\end{proof}

The previous result applies, in particular, to the classifying algebroid $A_\theta$ of the coframe. Hence, if $M$ is connected, we see that for any $p,q\in M$ the Lie algebras $\X(M,\theta)_p$ and $\X(M,\theta)_q$ are isomorphic. In the case of coframes, the following result is classical (see Theorem 3.2 of \cite{Kobayashi}):

\begin{theorem}
\label{thm:Lie:symmetries}
Let $(M,\theta,h)$ be a realization of a Cartan's problem. Then its group of symmetries $\Diff(M,\theta,h)$ is a (finite dimensional) Lie group of tranformations of $M$ with Lie algebra $\X_c(M,\theta,h)\subset\X(M,\theta,h)$ the subspace generated by the complete infinitesimal symmetries of the realization. 
\end{theorem}

\begin{proof}
We can assume that $M$ is connected. A classical theorem of Palais (Theorem IV.III of \cite{Palais}) states that any group $G\subset \Diff(M)$ such that the Lie subalgebra $\gg\subset\X(M)$ generated by the complete vector fields $X$ whose flows $\phi^t_X\in G$ is finite dimensional, is actually a Lie group of tranformations of $M$ with Lie algebra $\gg$. Hence, by Proposition \ref{prop:inf:sym:coframe}, it is enough to prove that for any $p\in M$ the restriction map $\X(M,\theta,h)\to \X(M,\theta,h)_p$ is injective. This follows from the following lemma:

\begin{lemma}
Let $X\in\X(M,\theta,h)$ be an infinitesimal symmetry and assume that $X$ vanishes at some point $p\in M$. Then $X\equiv 0$.
\end{lemma}

To prove this lemma, it is enough to to prove that the zero set of $X\in \X(M,\theta,h)$ is both open and closed. It is obviously closed, and to prove that it is open let $q\in M$ be such that $X_q=0$. If $\{X_i\}$ are the linearly independent vector fields dual to the coframe $\{\theta^i\}$, then:
\[ \Lie_X\theta^i=0\quad \Leftrightarrow \quad \Lie_X X_i=0.\]
Hence,  we have that:
\[ X_q=0 \text{ and } \Lie_{X_i}X=0.\]
Since the $X_i$ are everywhere linearly independent, we conclude that $X$ vanishes in a neighborhood of $q$.
\end{proof}

We also have the following semi-global symmetry property of a realization:

\begin{corol}[Theorem A.3 of \cite{Bryant}]
Let $(n, X, C^k_{ij}, F_i)$ be the initial data of a realization problem with associated Lie algebroid $A \to X$. Let $L \subset X$ be a leaf of $A$ whose isotropy Lie algebra is $\gg$, and let $G$ be any Lie group with Lie algebra $\gg$. Then over any contractible set $U \subset L$, there exists a principal $G$-bundle $h: M \to U$ and a $G$-invariant coframe $\theta$ of $M$ such that $(M, \theta, h)$ is a realization of the Cartan's problem. Moreover, this realization is locally unique up to isomorphism.
\end{corol}

\begin{proof}
This is an immediate consequence of our classification result, Theorem \ref{thm: classification coframes}. In fact, since $U$ is contractible and the restriction of $A$ to $U$ is transitive, it follows that $A\vert_U$ is integrable. Moreover, for any Lie group $G$ with Lie algebra $\gg$, there exists a Lie groupoid $\G$ integrating $A|_{U}$ whose isotropy Lie group is isomorphic to $G$. The restriction of the Maurer-Cartan form of $\G$ to any $\s$-fiber furnishes the desired (locally unique) $G$-invariant coframe.  
\end{proof}

We can also use Theorem \ref{thm: classification coframes} to explain the relationship between the classifying algebroid $A\to X$ of a Cartan realization problem and the classifying algebroid of a coframe $\theta$ associated with a realization $(M,\theta,h)$:

\begin{corol}
\label{cor:subalgbrd}
Let $(n,X,C_{ij}^k,F_i)$ be a Cartan realization problem with associated classifying algebroid $A$ and let $(M,\theta,h)$ be a connected realization. Then $h(M)$ is an open subset of an orbit of $A$ and there is a Lie algebroid morphism from the Lie subalgebroid $A|_{h(M)}\subset A$ to $A_\theta\to X_\theta$ (the classifying algebroid of the coframe $\theta$) which is a fiberwise isomorphism that covers a surjective submersion. 
\end{corol}

\begin{proof}
Since $M$ is connected, $h(M)$ is an open subset of an orbit of $A$. Also, we have a two leg diagram:
\[
\xymatrix{ & M\ar[dl]_h\ar[dr]^{\kappa_\theta}\\
h(M)\ar@{-->}[rr]_{\psi}&& X_\theta} 
\]
Observe that if $h(p)=h(q)$ then Theorem \ref{thm: classification coframes} gives a locally defined diffeomorphism $\phi:M\to M$, such that $\phi(p)=q$ and $\phi^*\theta=\theta$. In other words, $p$ and $q$ are locally formally equivalent so that we must have $\kappa_\theta(p)=\kappa_\theta(q)$. This shows that we have a well defined map $\psi:h(M)\to X_\theta$ that makes the diagram above commutative. Since $\kappa_\theta$ is a surjective submersion, we conclude that $\psi:h(M)\to X_\theta$ is also a surjective submersion.

Finally, observe that the coframe yields Lie algebroid maps
\[
\xymatrix{ & TM\ar[dl]\ar[dr]\\
A|_{h(M)}\ar@{-->}[rr]&& A_\theta} 
\]
which cover the two legs in the diagram above, and which are fiberwise isomorphism. Hence, we conclude that there is a Lie algebroid morphism $A|_{h(M)}\to A_\theta$,
over the map $\psi:h(M)\to X_\theta$, which is a fiberwise isomorphism.
\end{proof}

\begin{remark}
\label{rmk: extra invariants}
The corollary above can be interpreted as saying that the realization problem associated to a coframe is the one that has the least amount of invariant functions involved (and thus the largest symmetry group). On the other hand, a general realization problem deals with a coframe along with extra invariants that must be preserved.
\end{remark}

The following example illustrates the point made in this remark. 

\begin{example}
Suppose that $\Psi: G \times X \to X$ is an action of an $n$-dimensional Lie group $G$ on a manifold $X$, and let  $\psi: \gg \to \X(X)$ be the associated infinitesimal action. Let us denote by $A = \gg\ltimes X$ the corresponding transformation Lie algebroid. Since $A$ is a trivial vector bundle, it determines a realization problem. In fact, if $e_1, \ldots, e_n$ is a basis of $\gg$ such that 
\[[e_i,e_j] = C^k_{ij}e_k \text{ and } F_i = \psi(e_i),\]
then $(n,X,C^k_{ij},F_i)$ is the initial data to a realization whose classifying Lie algebroid is $A$. 

It then follows from Theorem \ref{thm: classification coframes} that every solution to this realization problem is locally equivalent to a neighborhood of the identity in $(G, \wmc^i, \Psi_x)$, where $\wmc^i$ are the components of the Maurer-Cartan form on $G$ with respect to the basis $e_1, \ldots, e_n$ of $\gg$ and 
\[\Psi_x: G \to X, \quad \Psi_x(g) = \Psi(g,x).\]
Thus, the infinitesimal symmetry algebra of a realization that corresponds to a point $x \in X$ is the isotropy Lie algebra of the infinitesimal action, i.e., $\X(G, \wmc^i, \Psi_x)_{\mathrm{Id}} = \ker \psi_x$. Moreover, a map $\phi: G \to G$ which preserves the Maurer-Cartan form is an equivalence, if and only if it leaves the action of $G$ on the orbit through $x$ invariant. In other words, it must be the right translation by an element of the isotropy subgroup $G_x$. 

On the other hand, the classifying Lie algebroid of the Maurer-Cartan coframe on $G$ is the Lie algebra $\gg$ itself. The symmetries of the Maurer-Cartan coframe on $G$ are the right translations by any element of $G$.
\end{example}

\section{Global Equivalence}                                     %
\label{sec:global}                                               %

We now turn to global questions related to equivalence of coframes. Namely, we will give our solution to the Globalization and to the Global Equivalence Problems

\subsection{The Globalization Problem}                           %
\label{subsec: Full Realizations}                                %

Recall that the \emph{globalization problem} (see Problem \ref{prob:globalization}) asks if two germs of coframes $\theta_0$ and $\theta_1$ which solve the sames realization problem are germs of the same global realization. 

We recall that we say that a Lie algebroid $A$ is \emph{weakly integrable} if the restriction of $A$ to any orbit is integrable. Then we have:

\begin{theorem}
\label{thm: globalization}
Consider a Cartan problem whose associated classifying Lie algebroid $A\to X$ is weakly integrable. Then $(M_0,\theta_0,h_0)$ and $(M_1,\theta_1,h_1)$ are germs of the same global connected realization $(M,\theta, h)$ if and only if they correspond to points on $X$ in the same orbit of $A$.
\end{theorem}

\begin{proof}
Given the two germs $(M_0,\theta_0,h_0)$ and $(M_1,\theta_1,h_1)$, suppose that they correspond to points $x_0$ and $x_1$ on the same orbit $L$ of $A$: $h_0(m_0)=x_0$ and $h_1(m_1)=x_1$. Since $A$ is weakly integrable, there exists a Lie groupoid $\G_L$ integrating $A|_L$ and the $\s$-fiber at $x_0$ contains a point $g$ with $\t(g) = x_1$. Thus, the germ of $(M_0,\theta_0,h_0)$ at $m_0$ can be identified with the germ of $(\s^{-1}(x_0),\wmc,\t)$ at $\mathbf{1}_{x_0}$ and the germ of $(M_1,\theta_1,h_1)$ at $m_1$  can be identified with the germ of $(\s^{-1}(x_0),\wmc,\t)$ at $g$. We conclude that $(M_0,\theta_0,h_0)$ and $(M_1,\theta_1,h_1)$ are both germs of the realization $(\s^{-1}(x_0), \wmc, \t)$.

Conversely, suppose that there exists a connected realization $(M,\theta,h)$ such that $(M_0,\theta_0,h_0)$ and $(M_1,\theta_1,h_1)$ are the germs of $\theta$ at points $p_0$ and $p_1$ of $M$, respectively. Let $\gamma$ be a curve joining $p_0$ to $p_1$ and cover it by a finite family $U_1, \ldots, U_k$ of open sets of $M$ with the property that the restriction of $\theta$ to each $U_i$ is equivalent to the restriction of the Maurer-Cartan form to an open set of some $\s$-fiber of $\G$, i.e., $\theta |_{U_i} = \phi_i^{\ast}\wmc$ for some diffeomorphism $\phi_i: U_i \to \phi_i(U_i) \subset \s^{-1}(x_i)$. 

We proceed by induction on the number of open sets needed to join $p_0$ to $p_1$. Suppose that both $p_0$ and $p_1$ belong to the same open set $U_1$. Then $\phi_1(p_0)$ and $\phi_1(p_1)$ are both on the same $\s$-fiber. Thus, $h(p_0) = \t \circ \phi_1(p_0)$ belongs to the same orbit as $h(p_1) = \t \circ \phi_1(p_1)$. Now assume that the result is true for $k-1$ open sets. Then any point $q$ in $U_{k-1}$ is mapped by $h$ to the same orbit of $h(p_0)$. Let $q$ be a point in $U_{k-1} \cap U_{k}$. Then, on one hand, since $q$ and $p_1$ belong to $U_k$ it follows that $h$ maps them both to the same orbit of $A$. On the other hand, by the inductive hypothesis, it follows that $h$ also maps $p_0$ and $q$ to the same orbit of $A$.
\end{proof}

\begin{remark}
The proposition above may not be true when $A$ is not weakly integrable. The problem is that in this case, the global object associated to $A_L$ is only a topological groupoid which is smooth only in a neighborhood of the identity section. Thus, if $x, y \in X$ are points in the same orbit which are ``too far away'', then we might not be able to find a \emph{differentiable} realization ``covering'' both points at the same time. 
\end{remark}

Motivated by this remark, it is natural to consider the problem of existence of realizations $(M, \theta,h)$ of a Cartan's problem, such that the image of $h$ is the whole leaf of the classifying Lie algebroid.

\begin{definition}
A connected realization $(M, \theta, h)$ is called \textbf{full} if $h$ is surjective onto the orbit of $A$ that it ``covers''. 
\end{definition}

The paradigmatic example of a full realization is obtained as follows.

\begin{example}
\label{ex: full}
Assume that the classifying Lie algebroid $A$ is weakly integrable. Then the $\s$-fibers of the Weinstein groupoid $\G(A)$ are smooth manifolds
and carry (the restriction of) the Maurer-Cartan form. Each triple $(\s^{-1}(x),\wmc,\t)$ is a full realization.
\end{example}

In fact, there is a partial converse to the previous example. To explain this we also need the following:

\begin{definition}
A realization $(M, \theta, h)$ of a Cartan's problem is said to be \textbf{complete} if it is a full realization and every local equivalence $\phi:U\to V$, defined on open sets $U,V\subset M$ can be extended to a global equivalence $\tilde{\phi}:M\to M$, such that $\tilde{\phi}|_U=\phi$.
\end{definition}

\begin{example}[(Example \ref{ex: full} continued)]
By Corollary \ref{symmetrywmc}, every local equivalence of the realization $(\s^{-1}(x),\wmc,\t)$ is the restriction of right translation by same element $g\in\G(A)$. Hence, this is an example of a complete realization.
\end{example}

Now we can state the following characterization (an analogous situation occurs in \cite{Schwachhofer:integration}):

\begin{prop}
Let $A \to X$ be the classifying Lie algebroid of a Cartan's realization problem, and let $L \subset X$ be an orbit of $A$. Then there exists a complete realization over $L$ if and only if the restriction $A\vert_L$ is integrable.
\end{prop}

\begin{proof}
Assume that $A\vert_L$ is integrable by $\G \tto L$. Then, as observed above, for any $x \in L$, the realization $(\s^{-1}(x),\wmc,\t)$ is complete.

Conversely, let $(M,\theta,h)$ be a complete realization which covers $L$. By Theorem \ref{thm:Lie:symmetries}, the group of symmetries $G=\Diff(M,\theta,h)$ is a Lie group of transformations of $M$ which fixes the fibers of $h:M\to L$. If $x,y\in M$ are such that $h(x)=h(y)$ then, by Theorem \ref{thm: classification coframes}, there exists a local equivalence $\phi:U\to V$ such that $\phi(x)=y$. By the completeness of the realization, $\phi$ can be extended to a global equivalence $\tilde{\phi}\in G$. We conclude that $G$ acts transitively on the fibers of $h$. 

Now we claim that the action of $G$ on $M$ is free. This will follow by showing that the fixed point set of a symmetry $\phi:M\to M$ is both open and closed. It is obviously closed. In order to prove that it is also open, we consider the vector fields $\{X_i\}$ dual to the coframe $\{\theta^i\}$ and denote by $\phi^t_{X_i}$ their (local) flows. Then, since $\phi$ is a symmetry, we have:
\[ \phi_*X_i=X_i\quad \Leftrightarrow \quad \phi\circ\phi^t_{X_i}=\phi^t_{X_i}\text{ (whenever defined)}.\]
Therefore, if $p\in M$ is some fixed point of $\phi$, it follows that $\phi^t_{X_i}(p)$ is also a fixed point of $\phi$. Since the $X_i$ are linearly independent everywhere, it follows that there is neighborhood of $p$ consisting of fixed points. This shows that the fixed point set of $\phi$ is open, and the claim follows.

We conclude that the $G$-action on $M$ is free and its orbits are the fibers of $h:M\to L$, so that we have a principal bundle
\[\xymatrix{M \ar[d]_{h} & \ar@(dr,ur)@<-6ex>[]_{G}\\
L}
\]
The Atiyah algebroid of this principal bundle is isomorphic to $A\vert_L$. Hence, $A\vert_L$ is integrable as claimed.
\end{proof}

\subsection{The Global Classification Problem}                   %
\label{subsec: Global Results}                                   %

We now turn to the global classification of coframes (Problem \ref{prob:global:equivalence}). For the remainder of this section, $(n, X, C^k_{ij}, F_i)$ denotes a realization problem with classifying Lie algebroid $A \to X$ and $(M,\theta,h)$ is a realization.

The global equivalence of realizations of a Cartan's problem is much more delicate than the question of local equivalence. The reason is that the classifying Lie algebroid will not distinguish between a realization and its covering. For example, Olver in \cite{Olver:Lie} constructs a simply connected manifold with a Lie algebra valued Maurer-Cartan form which cannot be globally embedded into any Lie group, but which is locally equivalent at every point to an open set of a Lie group.

We have the following theorem (the special cases of coframes of rank zero or maximal rank were considered by Olver in \cite{Olver}):

\begin{theorem} 
\label{thm: global equivalence up to cover}
Let $(M, \theta, h)$ be a realization of a Cartan problem and suppose that the classifying Lie algebroid $A \to X$ is weakly integrable. Then $M$ is globally equivalent up to a cover to an open set of an $\s$-fiber of a groupoid $\G$ integrating $A$.
\end{theorem}

\begin{proof}
Let $L\subset X$ be the orbit containing $h(M)$ and choose a Lie groupoid $\G_L$ integrating $A|_L$. Also, let $\mathcal{D}$ be the distribution on $M {_h\hskip-3 pt}\times_{\t} \G_L$ considered in the proof of Theorem \ref{Universal} and let $N$ be a maximal integral manifold of $\mathcal{D}$. If $\pi_M:  M {_h\hskip-3 pt}\times_{\t} \G_L \to M$ and $\pi_{\G}: M {_h\hskip-3 pt}\times_{\t} \G_L \to \G_L$ denote the natural projections, then $\pi_{\G}(N)$ is totally contained in a single $\s$-fiber of $\G_L$, say, $\s^{-1}(x)$. Moreover, the restrictions of the projections to $N$, $\pi_M: N \to M$ and $\pi_{\G}: N \to \s^{-1}(x)$ are local diffeomorphisms.

We claim that $N$, equipped with  the coframe $\pi_M^{\ast}\theta = \pi_{\G}^{\ast}\wmc$, is a common realization cover of $M$ and an open set of an $\s$-fiber of $\G_L$. To show this, all we must show is that $\pi_M(N) = M$. In fact, if this is true, $M$ will be globally equivalent up to cover to $\pi_{\G}(N) \subset \s^{-1}(x)$.

To prove this, suppose that $\pi_M(N)$ is a proper submanifold of $M$ and let $p_0 \in M - \pi_M(N)$ be a point in the closure of $\pi_M(N)$ (remember that $\pi_M\vert_{_N}$ is an open map). Then, by the local universal property of Maurer-Cartan forms, there is an open set $U$ of $p_0$ in $M$ and a diffeomorphism $\phi_0: U \to \phi(U) \subset \s^{-1}(h(p_0))$ such that $\phi_0^{\ast}\wmc = \theta$ and $\phi_0(p_0) = \mathbf{1}_{h(p_0)}$. It follows that the graph of $\phi_0$ is also an integral manifold $N_0$ of the distribution $\mathcal{D}$ on $M {_h\hskip-3 pt}\times_{\t} \G$ which passes through $(p_0, \mathbf{1}_{h(p_0)})$. Now let $p = \pi_M(p,g)$ be any point in $U \cap \pi_M(N)$ where $(p,g) \in N$. Then $\phi_0(p) \in \G$ is an arrow from $h(p_0)$ to $h(p)$ and $g$ is an arrow from $x$ to $h(p)$ and thus $g_0 = \phi_0(p)^{-1} \cdot g$ is an arrow from $x$ to $h(p_0)$, i.e.,
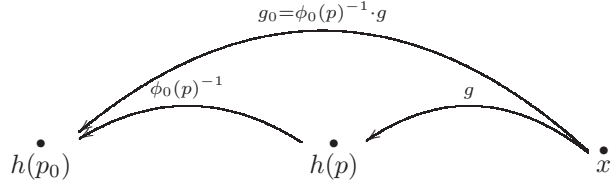
\[\xymatrix{ \stackrel{\bullet}{h(p_0)}  &&& \stackrel{\bullet}{h(p)} \ar@/_0.8cm/[lll]_{\phi_0(p)^{-1}} &&& \stackrel{\bullet}{x} \ar@/_0.8cm/[lll]_{g} \ar@/_1.8cm/[llllll]_{g_0 = \phi_0(p)^{-1} \cdot g}}\]

Now, by virtue of the invariance of the Maurer-Cartan form, the manifold 
\[R_{g_0} N_0 = \set{(\bar{p},\bar{g}\cdot g_0) : (\bar{p},\bar{g}) \in N_0}\]
is an integral manifold of $\mathcal{D}$. But then, the point $(p,g) = (p, \phi_0(p) \cdot g_0)$ belongs to $R_{g_0} N_0$ and to $N$, and thus, by the uniqueness and maximality of $N$ it follows that $N$ contains $R_{g_0} N_0$. But then,  $(p_0, \phi_0(p_0) \cdot g_0)$ is a point of $N$ which projects through $\pi_M$ to $p_0$, which contradicts the fact that $p_0 \in M -\pi_M(N)$, proving the theorem.
\end{proof}

We can now deduce the special cases where the rank of the coframe is either zero or maximal:

\begin{corol}[Theorem 14.28 of \cite{Olver}]
Let $\theta$ be a coframe of rank $0$ on a manifold $M$. Then $M$ is globally equivalent, up to covering, to an open set of a Lie group.
\end{corol}

\begin{proof}
If $\theta$ has rank $0$, the classifying Lie algebroid of the coframe has base $X= \set{*}$, i.e., it is a Lie algebra, so it is integrable by a Lie group. It follows that $M$ is globally equivalent, up to covering, to an open set of this Lie group.
\end{proof}

\begin{corol}[Theorem 14.30 of \cite{Olver}]
Let $(M_1, \theta_1, h_1)$ and $(M_2, \theta_2, h_2)$ be full realizations of rank $n$ over the same leaf $L \subset X$ of the classifiying Lie algebroid $A\to X$. Then $M_1$ and $M_2$ are realization covers of a common realization $(M, \theta, h)$.
\end{corol}

\begin{proof}
The assumptions of the corollary guarantee that $L$ is a $n$-dimensional leaf of $A$ for which $h_1(M_1) = L = h_2(M_2)$. Since $A$ has rank $n$, it follows that the anchor of the restriction of $A$ to $L$ is injective, and thus $A|_L$ is integrable (see \cite{CrainicFernandes}). Moreover, its Lie groupoid $\G$ has discrete isotropy, so its target map restricts to a local diffeomorphism on each source fiber.

It follows that we can make $M=L$ into a realization by equipping it with the pullback of the Maurer-Cartan form by the local inverses of $\t$. More precisely, let $U$ be an open set in $\s^{-1}(x)$ for which the restriction of $\t$ is one-to-one and let $V = \t(U)$ be the open image of $U$ by $\t$. Define $\theta_U$ to be the coframe on $V$ given by
\[\theta_U = (\t\vert_U^{-1})^{\ast}\wmc.\]
Then $\theta_U$ is the restriction to $V$ of a globally defined coframe on $L$. In fact, suppose that $\bar{U}$ is another open set of $\s^{-1}(x)$ for which $\t\vert_{\bar{U}}$ is one-to-one and such that  $\bar{V} = \t(\bar{U})$ intersects $V$. Denote by $\theta_{\bar{U}}$ the coframe on $\bar{V}$ defined by 
\[\theta_{\bar{U}} = (\t\vert_{\bar{U}}^{-1})^{\ast}\wmc.\]
We will show that $\theta_U$ and $\theta_{\bar{U}}$ coincide on the intersection $V \cap \bar{V}$. After shrinking $U$ and $\bar{U}$, if necessary, we may assume that $V = \bar{V}$. But then, since $\G$ is \'etale, it follows that the isotropy group $\G_x$ is discrete, which implies that $U$ is the right translation of $\bar{U}$ by an element $g\in \G_x$, i.e., $U = R_g(\bar{U})$. Thus,
\begin{align*}
\theta_U & = (R_g \circ \t\vert_{\bar{U}}^{-1})^{\ast}\wmc \\
&=  (\t\vert_{\bar{U}}^{-1})^{\ast} (R_g^{\ast}\wmc)\\
&=(\t\vert_{\bar{U}}^{-1})^{\ast}\wmc =\theta_{\bar{U}},
\end{align*}
and $\theta$ is a well defined global coframe on $L$.

Now, since both $M_1$ and $M_2$ are globally equivalent, up to covering, to open sets in $\s^{-1}(x)$, it follows that the surjective submersions $h_i: M_i \to L$ are realization covers. 
\end{proof}

\begin{remark}
There is a different, more direct, argument to see that the $h_i$'s are realization covers.  In fact, the coframe $\theta_1$ on $M_1$ is locally the pullback of the Maurer-Cartan form on an $\s$-fiber of $\G$ by some locally defined diffeomorphism $\phi_1 : W_1 \subset M_1 \to \s^{-1}(x)$. After shrinking $W_1$, we may assume that the restriction of $\t$ to $U_1 = \phi_1(W_1)$ is a diffeomorphism. But then, 
\begin{align*}
h_1^{\ast} \theta & = (\t \circ \phi_1)^{\ast} \theta \\
& = \phi_1^{\ast} (\t^{\ast} \circ (\t\vert_U^{-1})^{\ast} \wmc)\\
& = \phi_1^{\ast}\wmc= \theta_1.
\end{align*}

Obviously, the very same argument can be given to show that $h_2$ is also a realization cover. This can be summarized by the diagram:
\[\xymatrix{& \s^{-1}(x) \ar[dd]^{\t} & \\
M_1 \ar[ur]^{\phi_1} \ar[dr]_{h_1} & & M_2 \ar[ul]_{\phi_2} \ar[dl]^{h_2}\\
& L}
\]
\end{remark}

\subsection{Cohomological Invariants of Geometric Structures}  %
\label{subsec:cohomol:invariants}                                %

The classifying Lie algebroid $A$ of a fixed geometric structure should be seen as a basic invariant of global equivalence up to covering of the structure. In fact, we have the following result:

\begin{prop}
\label{prop:algbrd:cover}
Let $\theta$ be a fully regular coframe on $M$ and let $\bar{\theta}$ be an arbitrary coframe on $\bar{M}$. If $(M,\theta)$ and $(\bar{M}, \bar{\theta})$ are globally equivalent, up to covering, then
\begin{enumerate}
\item[(i)] $\bar{\theta}$ is a fully regular coframe on $\bar{M}$, and
\item[(ii)] the classifying Lie algebroids of $\theta$ and $\bar{\theta}$ are isomorphic.
\end{enumerate}
\end{prop}

\begin{proof}
In order to prove this proposition, it is enough to consider the case where $\theta$ is a fully regular coframe on $M$, $\bar{\theta}$ is a coframe on another manifold $\bar{M}$ and $\pi:\bar{M}\to M$ is a surjective local diffeomorphism which preserves the coframes. 

First we need:

\begin{lemma}
\label{lem:inv:cover}
The invariant functions of the two coframes are in 1:1 correspondence: $\Inv(\bar{\theta})=\pi^*\Inv(\theta)$.
\end{lemma}

\begin{proof}[Proof of Lemma \ref{lem:inv:cover}]
Clearly, since $\pi:\bar{M}\to M$ is a cover, the locally defined self equivalences of $(M,\theta)$ and $(\bar{M},\bar{\theta})$ correspond to each other. Hence the result follows.
\end{proof}
\vskip 10 pt

It follows that if $\theta$ is a fully regular coframe so is the coframe $\bar{\theta}$. Let us denote by $A_\theta \to X_\theta$ and $A_{\bar{\theta}}\to X_{\bar{\theta}}$ the classifying algebroids of $(M,\theta)$ and $(\bar{M},\bar{\theta})$, respectively. Also, we denote by $\kappa_\theta:M\to X_\theta$ and $\kappa_{\bar{\theta}}:\bar{M}\to \bar{\theta}$ the corresponding classifying maps, so $(M,\theta,\kappa_\theta)$ and $(\bar{M},\bar{\theta},\kappa_\theta)$ are realizations of $A_\theta$ and $A_{\bar{\theta}}$. Since $(\bar{M},\bar{\theta},\kappa_\theta\circ\pi)$ is also a realization of $A_\theta$ it follows from Corollary \ref{cor:subalgbrd} that there is an algebroid morphism from $A_\theta$ to $A_{\bar{\theta}}$, which is injective on the fibers and covers a local diffeomorphism $X_\theta\to X_{\bar{\theta}}$ making the following diagram commutative:
\[
\xymatrix{
M\ar[d]_{\kappa_\theta}& \bar{M}\ar[l]_{\pi}\ar[d]^{\kappa_{\bar{\theta}}}\\
X_\theta\ar[r]& X_{\bar{\theta}}
}
\]

We claim that the map $X_\theta\to X_{\bar{\theta}}$ is injective, so that $A_\theta$ and $A_{\bar{\theta}}$ are isomorphic. In fact, let $x,y\in X_\theta$ which are mapped to the same point in $z\in X_{\bar{\theta}}$. Then we can choose points $p,q\in M$ and $\bar{p},\bar{q}\in\bar{M}$ which are mapped to each other under the commutative diagram above. Then, by Theorem \ref{thm: classification coframes}, there exists a local equivalence $\bar{\phi}:\bar{M}\to \bar{M}$, such that  $\phi(\bar{p})=\bar{q}$. But then (after restricting to small enough open sets) $\bar{\phi}$ covers a local equivalence $\phi:M\to M$ such that $\phi(p)=q$. This means that $p$ and $q$ are local equivalent, so $x=\kappa_\theta(p)=\kappa_\theta(q)=y$, thus proving the claim.
\end{proof}

Even though the isomorphism class of the classifying Lie algebroid does not distinguish coframes which are globally equivalent, up to covering, one can use its Lie algebroid cohomology to obtain invariants of coframes. To illustrate this point of view, we will now describe two invariants of coframes arising in this way.

Given a Lie algebroid $A$ we will denote by $(\Omega^\bullet(A),\d_A)$ the complex of $A$-forms where $\Omega^\bullet(A):=\Gamma(\wedge^{\bullet}A^*)$ and the differential is given by the formula:
\begin{multline*}
\d_A \omega (\al_0, \ldots \al_k) = \sum_{i=1}^k (-1)^i \sharp(\al_i)\cdot \omega(\al_0, \ldots, \hat{\al_i}, \ldots, \al_k) + \\
+ \sum_{0 \leq i < j \leq k} (-1)^{i+j} \omega([\al_i,\al_j],\al_0, \ldots, \hat{\al_i}, \ldots, \hat{\al_j},\ldots, \al_k).
\end{multline*}
The resulting cohomology $H^{\bullet}(A)$ is called the Lie algebroid cohomology of $A$. Every Lie algebroid morphism $\Phi:A \to B$ induces a chain map 
\[\Phi^*:(\Omega^{\bullet}(B), \d_B) \to (\Omega^\bullet(A),\d_A),\]
and, hence, also a map in cohomology $\Phi^*:H^\bullet(B)\to H^\bullet(A)$.

For a fully regular coframe $\theta$ on $M$ the Lie algebroid cohomology of its classifying algebroid $A_\theta$ can be described in terms of \emph{invariants forms}. Let us call $\omega\in\Omega^k(M)$ a \textbf{k-invariant form} if for any locally defined equivalence of the coframe $\phi:M\to M$ one has:
\[ \phi^*\omega=\omega. \]
Since the differential of an invariant form is again an invariant form, the subspace of invariant forms $\Omega^{\bullet}_\theta(M)$ is a subcomplex of the de Rham complex.

\begin{definition}
The \textbf{invariant cohomology of $(M,\theta)$}, denoted $H^*_{\theta}(M)$ is the cohomology of the complex of invariant forms.
\end{definition}

Now we note that:

\begin{prop}
\label{prop:inv:cohom}
Let $(M,\theta)$ be a fully regular coframe. Then the Lie algebroid cohomology $H^*(A_{\theta})$ coincides with the invariant cohomology $H^*_\theta(M)$.
\end{prop}

\begin{proof}
This is simply a consequence of the fact that $\theta^*:\Omega^\bullet(A_\theta)\to\Omega^\bullet(M)$ is injective and has image precisely the complex of invariant forms. In fact, a form $\omega\in\Omega^k(M)$ is the pullback by $\theta$ of a section $\varphi$ of $\wedge^kA^*_\theta$ iff when we write $\omega$ in terms of the coframe $\theta$ we have:
\[\omega = \sum_{i_1 < \cdots < i_k}a_{i_1,\ldots,i_k}\circ\kappa_\theta~\theta^{i_1}\wedge \cdots \wedge \theta^{i_k},\]
for some smooth functions $a_{i_1,\ldots,i_k}\in C^\infty(X_\theta)$.
\end{proof}

Since can view the coframe as a Lie algebroid map $\theta:TM\to A_\theta$ (over the classifying map $\kappa_\theta:M\to X_\theta$), there is an induced map in cohomology $\theta^*: H^{\bullet}(A_\theta) \to H^{\bullet}_{\mathrm{dR}}(M)$. By the previous proposition, a class belongs to the image of this map iff it can be represented by an invariant $k$-form. In fact, under the isomorphism $H^*(A_{\theta})\simeq H^*_\theta(M)$ the map $\theta^*: H^{\bullet}(A_\theta) \to H^{\bullet}_{\mathrm{dR}}(M)$ corresponds to the map $H^*_\theta(M) \to H^{\bullet}_{\mathrm{dR}}(M)$ induced by the forgetful map.

\begin{remark}
If every local equivalence can be extended to a global equivalence (i.e., if $(M,\theta,\kappa_\theta)$ is a complete realization) then $H^*_\theta(M)$ coincides with the invariant cohomology $H^\bullet_G(M)$ associated with the action of $G=\Diff(M,\theta)$, the group of symmetries of the coframe. If further $G$ is compact, then we have $H^\bullet_G(M)=H^{\bullet}_{\mathrm{dR}}(M)$. 
\end{remark}

From Proposition \ref{prop:algbrd:cover} we see that invariant cohomology is an invariant of global equivalence, up to covering.

\begin{corol}
\label{cor:equiv:cohomology}
Let $(M,\theta)$ and $(\bar{M},\bar{\theta})$ be fully regular coframes which are globally equivalent, up to covering. Then their invariant cohomologies are isomorphic $H_{\theta}^*(M)\simeq H^*_{\bar{\theta}}(\bar{M})$.
\end{corol}

The invariant cohomology of a coframe is the natural place where classes associated with the coframe live. We will study next one such example. 

\begin{remark}
\label{rem:cohom:realization}
There is also an obvious extension of these constructions and results to any realization $(M, \theta, h)$ of a Cartan problem, so one can define the invariant cohomology $H^*_{\theta,h}(M)$. This cohomology is invariant under global equivalence and if  $(M,\theta, h)$ is complete, and $G = \Diff(M,\theta,h)$, then it coincides with $H^\bullet_G(M)$.
\end{remark}

\subsection{The modular class}                                   %
\label{subsec:mod:class}                                             %

Lie algebroids have associated intrinsic characteristic classes (see \cite{Fernandes:holonomy,Crainic:Van,CrainicFernandes:exotic}). Therefore, for a fully regular coframe $\theta$ we can consider the characteristic classes of its classifying Lie algebroid $A_\theta$ and we obtain certain invariant cohomology classes which maybe called the characteristic classes of the coframe. The simplest of these classes is the \emph{modular class} (\cite{Weinstein:modular1,Weinstein:modular2}) which we recall briefly.

Given a Lie algebroid $A\to M$ the line bundle $L_A:=\wedge^{\text{top}}A\otimes\wedge^{\text{top}}T^*M$ carries a natural flat $A$-connection, which makes $L$ into a $A$-representation, defined by:
\[ \nabla_\al (\omega\otimes\nu)=\Lie_\al\omega\otimes\nu+\omega\otimes\Lie_{\rho(\al)}\nu\quad (\al\in\Gamma(A)). \]
When $L_A$ is orientable, so that it carries a nowhere vanishing section $\mu\in\Gamma(L_A)$, we have:
\[ \nabla_\al \mu=\langle c_\mu,\al\rangle\mu,\quad \forall \al\in\Gamma(A),\]
for a 1-form $c_\mu\in\Omega^1(A)$. One checks easily that $c_\mu$ is $\d_A$-closed and one calls $c_\mu$ the \textbf{modular cocycle}  of $A$ relative to the nowhere vanishing section $\mu$. If $\mu'=f \mu$ is some other non-vanishing section, then one finds that:
\[ c_{\mu'}=c_\mu+\d_A \log |f|. \]
so the cohomology class $\modular(A):=[c_\mu]\in H^1(A)$ does not depend on the choice of global section $\mu$.  One calls $\modular(A)$ the \textbf{modular class} of $A$. When $L_A$ is not orientable, one repeats the construction for $L_A\otimes L_A$ and defines $\modular(A)$ to be one half the cohomology class associated with the line bundle  $L_A\otimes L_A$.

\begin{definition}
The \textbf{modular class of a fully regular coframe} is the invariant cohomology class $\modular(\theta)\in H^1_{\theta}(M)$ which under the natural isomorphism $H^\bullet_{\theta}(M)\simeq H^\bullet (A_\theta)$ corresponds to the class $\modular(A_\theta)$.
\end{definition}

It should be clear that if $(M,\theta)$ and $(\bar{M},\bar{\theta})$ are fully regular coframes which are globally equivalent, up to covering, then under the natural isomorphism $H_{\theta}^*(M)\simeq H^*_{\bar{\theta}}(\bar{M})$ given by Corollary \ref{cor:equiv:cohomology} the modular classes correspond to each other.

Our next proposition leads to a geometric interpretation of the modular class of a coframe:

\begin{prop}
\label{prop:mod:class}
Let $(M,\theta)$ be a fully regular coframe with symmetry Lie algebra $\gg:=\X(M,\theta)$ of dimension $k$. Let $\{\xi_1,\dots,\xi_k\}$ be a basis of $\gg$ and let $\mu_G$ be an invariant k-form which restricts to a volume form on each $\gg$-orbit. Then:
\[ \modular(\theta)=-[\d\log|\langle\mu_G,\xi_1\wedge\cdots\wedge\xi_k\rangle|]\in H^1_\theta(M). \]
\end{prop}

\begin{proof}
For a transitive algebroid the bundle $\ker\rho$ is a natural $A$-representation and we have a natural isomorphism of $A$-representations $\wedge^{\text{top}}\ker\rho\simeq L_A$. By Remark \ref{rem:transitive}, the classifying Lie algebroid $A_\theta\to X_\theta$ is always transitive. The classifying map $k_\theta:M\to X_\theta$, whose fibers coincide with the $\gg$-orbits, give an identification $k_\theta^*\ker\rho\simeq \ker\d k_\rho$. Hence, we can identify invariant sections of $\ker\d k_\rho$ with sections of $\ker\rho$, and it follows that any invariant k-form $\mu_G$ which restricts to a volume form on each $\gg$-orbit determines a global section of $L_A^*$ (and, hence, of $L_A$). If $\{\al_1,\dots,\al_n\}$ denotes the canonical basis of $A$ (which we identify with the invariant vector fields dual to the coframe $\{\theta^1,\dots,\theta^n\}$), then under these identifications, we have that:
\[ \nabla_{\al_i} \mu_G =\Lie_{\al_i} \mu_G, \]
where on the right-hand side we view $\al_i$ as an invariant vector field on $M$. It follows that the modular cocycle relative to the section $\mu=\mu_G^*$ is given by:
\begin{align*}
c_\mu(\al_i)&=\frac{\langle c_\mu(\al_i)\mu_G,\xi_1\wedge\cdots\wedge\xi_k\rangle}{\langle\mu_G,\xi_1\wedge\cdots\wedge\xi_k\rangle}\\
&=\frac{\langle -\nabla_{\al_i}\mu_G,\xi_1\wedge\cdots\wedge\xi_k\rangle}{\langle\mu_G,\xi_1\wedge\cdots\wedge\xi_k\rangle}\\
&=\frac{\langle -\Lie_{\al_i}\mu_G,\xi_1\wedge\cdots\wedge\xi_k\rangle}{\langle\mu_G,\xi_1\wedge\cdots\wedge\xi_k\rangle}\\
&=-\frac{\Lie_{\al_i}\langle\mu_G,\xi_1\wedge\cdots\wedge\xi_k\rangle}{\langle\mu_G,\xi_1\wedge\cdots\wedge\xi_k\rangle}=-\langle \al_i,\d\log|\langle\mu_G,\xi_1\wedge\cdots\wedge\xi_k\rangle|\rangle
\end{align*}
This proves that the expression for the modular class in the statement of the proposition holds.
\end{proof}

The proposition shows, in particular, that the natural map $H^*_\theta(M) \to H^{\bullet}_{\mathrm{dR}}(M)$ sends the modular class $\modular(\theta)$ to zero. Thus, the modular class is an obstruction for a complete coframe to have a compact symmetry group, i.e.,
\begin{corol}
Let $(M,\theta)$ be a fully regular coframe. If $\theta$ is complete, and the symmetry group $\Diff(M, \theta)$ is compact, then $\modular(\theta) = 0$.
\end{corol}

Also, we conclude from Proposition \ref{prop:mod:class}  that:

\begin{corol}
The symmetry Lie algebra $\gg:=\X(M,\theta)$ is unimodular if and only if $\modular(\theta)=0$.
\end{corol}

\begin{proof}
One finds for any $\xi\in\gg$ that:
\begin{align*}
\Lie_\xi \log |\langle\mu_G,\xi_1\wedge\cdots\wedge\xi_k\rangle|&=\frac{\Lie_\xi \langle\mu_G,\xi_1\wedge\cdots\wedge\xi_k\rangle}{\langle\mu_G,\xi_1\wedge\cdots\wedge\xi_k\rangle}\\
&=\frac{\sum_{i=1}^k \langle\mu_G,\xi_1\wedge\cdots \wedge [\xi,\xi_i] \wedge\cdots\wedge\xi_k\rangle}{\langle\mu_G,\xi_1\wedge\cdots\wedge\xi_k\rangle}=\tr(\ad\xi)
\end{align*}
Hence, we conclude that $\log |\langle\mu_G,\xi_1\wedge\cdots\wedge\xi_k\rangle|$ is an invariant function if and only if $\gg$ is unimodular.
\end{proof}

\begin{remark}
More generally, if $(M, \theta, h)$ is a realization of a Cartan problem associated with a classifying Lie algebroid $A \to X$ one can also define its modular class by
\[ \modular(M, \theta, h):=\modular(A) \in H^1_{\theta,h}(M) \]
where $H^1_{\theta,h}(M)$ denotes the invariant cohomology (see Remark \ref{rem:cohom:realization}).  A version of Proposition \ref{prop:mod:class} also holds.
\end{remark}

\section{An Example: Surfaces of Revolution}
\label{sec: examples}

In this section, we consider the local classification of surfaces of revolutions up to isometries. We begin by deducing the structure equations of a generic surface of revolution following very closely the differential analysis presented in Chapter 12 of \cite{Olver}. However, since the local moduli space of surfaces of revolution is infinite dimensional (we will see that it depends on an arbitrary function which will be denoted by $H$), we will then restrict to a finite dimensional class of surfaces of revolution, namely those for which the moduli $H$ has a fixed value, and exhibit its classifying Lie algebroid. We use this example mainly to illustrate several of the concepts and results presented throughout the paper. In the sequel to this paper, where we discuss $G$-structures, we will discuss more sophisticated (and interesting examples), such as special symplectic manifolds.

Let $x,y,z$ denote the canonical coordinates on $\Rr^3$ and let $z = f(x)$ be a curve in the plane $\set{y = 0}$ which does not intersect the $z$-axis, i.e., such that $f(x) \neq 0$ for all $x \in \Rr$. Let $\Sigma$ be the surface obtained by rotating $f(x)$ around the $z$-axis, which we can parameterize using polar coordinates
\[ \left\{
\begin{array}{l}
x = r \cos v \\
y = r \sin v \\
z = f(r)
\end{array}\right. \]
where $r >0$.

The relevant coframes for this classification problem is the one which diagonalizes the metric $\mathrm{ds}^2$ induced by the Euclidean metric on $\Sigma \subset \Rr^3$, i.e., the coframes $\{\omega^1,\omega^2\}$ for which 
\[\mathrm{ds}^2 = (\omega^1)^2 + (\omega^2)^2.\]
In order to simplify the expression for such a coframe, we first perform the change of coordinates
\[u = \int \sqrt{1 + f^{\prime}(r)^2}\d r\]
with inverse $r = h(u)$, where $h(u) > 0$ for all $u$. With respect to the new coordinates $(u,v)$, the metric can be written as
\[\mathrm{ds}^2 = (\d u)^2 + h(u)^2(\d v)^2, \]
so any diagonalizing coframe for $\mathrm{ds}^2$ is given by 
\[ \left\{
\begin{array}{l}
\omega^1 = \d u \\
\omega^2 = h(u)\d v.
\end{array}\right. \]

Notice that any coframe which differs from $\{\omega^1, \omega^2\}$ by a rotation represents the same metric. This suggests that we should consider the coframe
\[ \left\{
\begin{array}{l}
\theta^1 = (\sin t) \omega^1 + (\cos t) \omega^2 \\
\theta^2 = (-\cos t) \omega^1 + (\sin t) \omega^2\\
\alpha = \d t
\end{array}\right. \]
on $\Sigma \times \mathrm{S}^1$, where $\alpha$ is the Maurer-Cartan form on $\mathrm{S}^1 \cong \mathrm{SO}(2)$. 

\begin{remark}
The passage from $M$ to $M\times G$ is the very first preliminary step in Cartan's equivalence method. The coframe in $M\times G$, obtained by adding the Maurer-Cartan forms, is called the \emph{lifted coframe} (see \cite{Olver}).
\end{remark}

Next, in order to obtain the structure equations and structure functions, we differentiate the one forms $\theta^1$, $\theta^2$ and $\alpha$ obtaining
\begin{equation} \label{eq: structure with torsion}
\left\{
\begin{array}{l}
\d\theta^1 = -\alpha \wedge \theta^2 + (\frac{h^{\prime}(u)}{h(u)}\cos t) \theta^1\wedge\theta^2 \\
\d\theta^2 = \alpha \wedge \theta^1 + (\frac{h^{\prime}(u)}{h(u)}\sin t) \theta^1\wedge\theta^2\\
\d\alpha = 0
\end{array}\right. \end{equation}
These structure equations depend explicitly on the coordinates $u$ and $t$. However, this can be fixed by replacing the form $\alpha$ by
\begin{equation}\label{eq: absorption of torsion}
\eta = \alpha - \frac{h^{\prime}(u)}{h(u)}((\cos t)\theta^1 + (\sin t)\theta^2) = \alpha - h^{\prime}(u)\d v.
\end{equation}

If we now rewrite \eqref{eq: structure with torsion} in terms of the new coframe we obtain
\[
\left\{
\begin{array}{l}
\d\theta^1 = -\eta \wedge \theta^2 \\
\d\theta^2 = \eta \wedge \theta^1\\
\d\eta = \kappa \theta^1 \wedge \theta^2
\end{array}\right. \]
where $\kappa = -h^{\prime \prime}(u)/h(u)$ is the Gaussian curvature of $\Sigma$.

\begin{remark}
The substitution of $\alpha$ by $\eta$ is another step in Cartan's equivalence method and is known as \emph{absorption of torsion} (see \cite{Olver}).
\end{remark}

The next step is to calculate the coframe derivative of $\kappa$. We obtain
\[\d \kappa = \kappa^{\prime}(u)((\sin t)\theta^1 - (\cos t)\theta^2).\]
However, since $\Sigma$ is a \emph{surface of revolution}, we know that its symmetry Lie group must be at least one dimensional. It follows that there can be at most two independent invariants, and since we already have two invariant functions, namely $t$ and $\kappa$, we conclude that $\kappa^{\prime}(u) = H(\kappa)$ for some $H \in \mathrm{C}^{\infty}(\Rr)$.

\begin{remark}\label{rmk: t as an invariant}
By its definition, $t$ is the angle between the orthogonal coframe at a point and the special coframe $\set{\d u, h(u)\d v}$ given in the adapted coordinates. Thus, imposing $t$ as an invariant is the same as considering surfaces of revolution up to isometries which fix the angle $t$. This can be geometrically interpreted as a restriction on the allowed symmetries of $\Sigma$. For a generic surface of revolution this is not really a restriction but, for example, when $H=0$, i.e., when the curvature of $\Sigma$ is constant, this imposes a true restriction on the symmetry group. 
\end{remark}

Finally, from \eqref{eq: absorption of torsion} we find that
\[\d t = \eta + J(\kappa)((\cos t)\theta^1 +(\sin t)\theta^2)\]
for some $J \in \mathrm{C}^{\infty}(\Rr)$. The relationship between the two functions $H$ and $J$ follows from imposing $\d^2\kappa = 0$. If we assume $\kappa'(u)\not=0$, i.e., that $H(\kappa)\not=0$, it yields
\begin{equation}\label{eq: differential relation}
J^{\prime}(\kappa)H(\kappa) = -\kappa - J(\kappa)^2.
\end{equation}
In conclusion, the full set of \textbf{structure equations for surfaces of revolution} is
 \begin{equation}\label{eq: structure equations for surfaces of revolution}
\left\{
\begin{array}{l}
\d\theta^1 = -\eta \wedge \theta^2 \\
\d\theta^2 = \eta \wedge \theta^1\\
\d\eta = \kappa \theta^1 \wedge \theta^2\\
\d \kappa = H(\kappa)((\sin t)\theta^1 - (\cos t)\theta^2)\\
\d t = \eta + J(\kappa)((\cos t)\theta^1 +(\sin t)\theta^2)
\end{array}\right. \end{equation}
where $H$ is an arbitrary smooth function and $J \in  \mathrm{C}^{\infty}(\Rr)$ is determined by \eqref{eq: differential relation}. Notice that $\d^2 = 0$ is a consequence of the equations above.

\begin{remark}[Continuing Remark \ref{rmk: t as an invariant}]
\label{rmk: t as an invariant:2}
When $\kappa'(u)=0$ the analysis above collapses and the structure equations for surfaces of revolution reduce to:
\begin{equation}
\label{eq:constant:curvature}
\left\{
\begin{array}{l}
\d\theta^1 = -\eta \wedge \theta^2 \\
\d\theta^2 = \eta \wedge \theta^1\\
\d\eta = \kappa \theta^1 \wedge \theta^2\\
\d \kappa = 0.
\end{array}\right. 
\end{equation}
\end{remark}

Now that we have found the structure equations for an arbitrary surface of revolution, we may immediately write the classifying Lie algebroid $A$ for such surfaces. As a vector bundle, $A \cong (\Rr \times \mathrm{S}^1)\times \Rr^3 \to \Rr \times \mathrm{S}^1$. Its Lie bracket is given on the constant sections $e_1,e_2$ and $e_3$ by 
\[
\begin{array}{l}
\left[e_1, e_2\right](\kappa, t) =  -\kappa e_3 \\
\left[e_1, e_3\right] (\kappa, t)= e_2 \\
\left[e_2, e_3\right] (\kappa, t)= -e_1
\end{array}\]
and its anchor is given by
\[
\begin{array}{l}
\sharp(e_1)(\kappa, t) =  H(\kappa)\sin t \frac{\partial}{\partial \kappa} + J(\kappa)\cos t \frac{\partial}{\partial t} \\
\sharp(e_2)(\kappa, t) =  -H(\kappa)\cos t \frac{\partial}{\partial \kappa} + J(\kappa)\sin t \frac{\partial}{\partial t} \\
\sharp(e_3)(\kappa, t) = \frac{\partial}{\partial t}.
\end{array}\]
Again, when $\kappa'(u)=0$ (cf.~Remarks \ref{rmk: t as an invariant} and \ref{rmk: t as an invariant:2}), the analysis collapses and we obtain an algebroid over a 1-dimensional manifold with zero anchor (i.e., a bundle of Lie algebras).

For any function $H(\kappa)$ the equations above determine a Lie algebroid. From Theorem \ref{thm: classification coframes} and the fact that the algebroid is transitive, we deduce that:

\begin{prop}
For any $\kappa_0 \in \Rr$ and any $H \in \mathrm{C}^{\infty}(\Rr)$ there exists a surface of revolution $\Sigma$ with $p \in \Sigma$ such that $\kappa(p) = \kappa_0$ and for which $\kappa^{\prime}(u)=H(\kappa)$. Moreover, any two such surfaces of revolution are locally isometric in a neighborhood of the points corresponding to $\kappa_0$.  
\end{prop}

We can also use the classifying Lie algebroid to describe the infinitesimal symmetries of surfaces of revolution (see Proposition \ref{prop:inf:sym:coframe} and Theorem \ref{thm:Lie:symmetries}). For this, we must distinguish the cases $\kappa'(u)\not=0$ and $\kappa'(u)=0$:

\begin{description}
\item[$H(\kappa)\not=0$] In this case, which is generic, it is easy to see that the anchor is surjective, and thus the symmetry Lie algebra is $1$-dimensional, which corresponds to rotation around the axis of revolution. 

\item[$H(\kappa)=0$] In this case, the surface of revolution has constant curvature. The orbits of the corresponding Lie algebroid are $0$-dimensional, and thus, the symmetry Lie algebra is $3$-dimensional. We have the following 3 possibilities:
\begin{center}
\begin{tabular}{||c|c|c||}\hline
$\mathfrak{sl}_{2}$   & \footnotesize{if $\kappa < 0$} & \footnotesize{Hyperbolic Geometry}  \\
\hline $\mathfrak{se}_{2}$ & \footnotesize{if $\kappa=0$}   & \footnotesize{Euclidean Geometry}        \\
\hline $\mathfrak{so}_{3}$ & \footnotesize{if $\kappa>0$} & \footnotesize{Spherical Geometry}  \\
\hline
\end{tabular}
\end{center}
\end{description}

\begin{remark}
Note that a surface of revolution which has non-constant $H(\kappa)$ which passes through $0$ \emph{cannot} appear as a realization of our Cartan's problem. The reason is that the corresponding coframe will not be fully regular.
\end{remark}

\begin{remark}
Strictly speaking the classifying Lie algebroid described above (and also the conclusions deduced from it) concern the coframe $\theta^1, \theta^2, \eta$, and not the surface itself. However, it will be shown in \cite{FernandesStruchiner} how to deal with Cartan's realization problem on finite type $G$-structures. In particular, the realization problem discussed here should be treated as a realization problem on the orthogonal frame bundle of the surface of revolution. 
\end{remark}

We remark that for each choice of a function $H$, the classifying Lie algebroid discussed above is the classifying Lie algebroid of a single coframe; namely the coframe $\theta^1,\theta^2, \eta$ on the orthogonal frame bundle of the surface of revolution. However, the (local) moduli space of surfaces of revolution is clearly infinite dimensional. Thus, in order to describe a class of surfaces of revolution, we will consider here only those for which $H$ is constant.

\begin{remark}
The condition that $H$ is constant can be seen a second order differential equation on the curvature tensor $R$ of the induced metric on the surface of revolution. In fact, using the coordinates $u,v$ introduced earlier, we can express the class of surfaces  being considered as the class of surfaces of revolution for which
\[\nabla_{\frac{\partial}{\partial u}}(\nabla R) = 0,\]
where $\nabla$ is the Levi-Civita connection of the surface. Since this condition is equivalent to $\kappa(u)$ being an affine map, we will denote such surfaces by \textbf{affinely  curved surfaces of revolution}.
\end{remark}

Note now that since we are not  prescribing a specific value for the constant $H$, we must consider it as a new invariant function for affinely curved surfaces of revolution which is subject to the condition $\d H = 0$. Similarly, $J$ is no longer determined and must also be added to our set of invariant functions. It then follows from \eqref{eq: differential relation} that whenever $H \neq 0$ (which we assume from now on)
\[\d J = -(\kappa + J^2)((\sin t) \theta^1 - (\cos t) \theta^2).\]

Thus, as a complete set of structure equations for affinely curved surfaces of revolution one obtains:
 \begin{equation}\label{eq: structure equations for affinely curved surfaces of revolution}
\left\{
\begin{array}{l}
\d\theta^1 = -\eta \wedge \theta^2 \\
\d\theta^2 = \eta \wedge \theta^1\\
\d\eta = \kappa \theta^1 \wedge \theta^2\\
\d \kappa = H(\kappa)((\sin t)\theta^1 - (\cos t)\theta^2)\\
\d t = \eta + J(\kappa)((\cos t)\theta^1 +(\sin t)\theta^2)\\
\d H = 0\\
\d J = -(\kappa + J^2)((\sin t) \theta^1 - (\cos t) \theta^2).
\end{array}\right. \end{equation}

Again, it is easily verified that these structure equations have as a formal consequence that $\d^2 = 0$, and thus they determine the classifying Lie algebroid $B$ for affinely curved surfaces of revolution. As a vector bundle,  $B \cong Y\times \Rr^3 \to Y$, where
\[Y = \set{(\kappa, H, J, t) \in \Rr^3 \times \mathrm{S}^1: H \neq 0}.\] 
Its Lie bracket is given on the constant sections $e_1,e_2$ and $e_3$ by 
\[
\begin{array}{l}
\left[e_1, e_2\right](\kappa, H, J, t) =  -\kappa e_3 \\
\left[e_1, e_3\right] (\kappa, H, J, t)= e_2 \\
\left[e_2, e_3\right] (\kappa, H, J, t)= -e_1
\end{array}\]
and its anchor is given by
\[
\begin{array}{l}
\sharp(e_1)(\kappa,H,J, t) =  H\sin t \frac{\partial}{\partial \kappa}  - (\kappa + J^2)\sin t\frac{\partial}{\partial J}+ J(\kappa)\cos t \frac{\partial}{\partial t} \\
\sharp(e_2)(\kappa,H,J, t) =  -H\cos t \frac{\partial}{\partial \kappa} +(\kappa + J^2)\cos t\frac{\partial}{\partial J} + J(\kappa)\sin t \frac{\partial}{\partial t} \\
\sharp(e_3)(\kappa,H,J, t) = \frac{\partial}{\partial t}.
\end{array}\]

Observe that the orbits of this Lie algebroid are 2-dimensional, so we can conclude that:

\begin{prop}
For any $\kappa_0, H_0, J_0$ with $H_0 \neq 0$ there exists an affinely curved surface of revolution whose structure invariants take the values $\kappa_0, H_0, J_0$. Moreover,
\begin{enumerate}
\item Every affinely curved surface of revolution has a one dimensional symmetry Lie algebra.
\item If two affinely curved surfaces of revolution are open submanifolds of a third affinely curved surface of revolution, then there invariant constant function $H$ must agree.
\end{enumerate}
\end{prop}

\appendix
\section{The classifying Lie algebroid of a single coframe}  
\label{append:single:coframe}                   

In this appendix we will use the jet bundle approach to give a coordinate free construction of the classifying Lie algebroid associated to a fully regular coframe $\theta$ on a manifold $M$. We begin by recalling some general constructions.

Let $N$ be a manifold and let $\pi: N \to M$ be a smooth map. We denote by $\Jet^1N\to M$ the bundle of 1-jets of sections of $\pi$. Let us denote by $\mathrm{Hom}(TN, \pi^*TM)$ the vector bundle over $N$ whose sections are bundle maps 
\[ \xymatrix{TN \ar[rr] \ar[rd]& & \pi^*TM\ar[dl]\\
& N, & }\]
and by $\Jet^1(\mathrm{Hom}(TN, \pi^*TM))$ the bundle of $1$-jets of sections of $\mathrm{Hom}(TN, \pi^*TM)$. After a choice of a (local) flat connection, we can identify a point of $\Jet^1(\mathrm{Hom}(TN, \pi^*TM))$ in the fiber over $x\in N$ with a pair $(q, l_q)$, where $q\in \mathrm{Hom}_x(TN, \pi^*TM)$ and
\[l_q: T_xN \to \mathrm{Hom}_x(TN, \pi^*TM)\]
is a linear map. It follows that we may define a bundle map 
\[\xymatrix{\Jet^1(\mathrm{Hom}(TN, \pi^*TM)) \ar[rr]^{\mathrm{Alt}} \ar[dr] && \mathrm{Hom}(\wedge^2TN, \pi^*TM) \ar[dl]\\
&N&}\]
given locally by
\[\mathrm{Alt}(q,l_q)(v\wedge w) = \frac{1}{2}l_q(v)(w) - l_q(w)(v).\]
It is easy to check that this formula does not depend on the choice of (local) flat connection, so the bundle map is well defined.

\begin{example}
If we take $M$ to be the real line $\Rr$ and $\pi$ to be a constant map, then the construction above yields a map 
\[\xymatrix{\Jet^1(T^*N) \ar[rr]^{\mathrm{Alt}} \ar[dr] & & \wedge^2T^*N \ar[dl]\\
& N &}\]
that satisfies
\[\mathrm{Alt}(\jet^1\al) = \d \al.\] 
\end{example}

Now, note that we may view $\pi_*$ as a section of $\mathrm{Hom}(TN, \pi^*TM)$. Thus, by applying $\mathrm{Alt}$, we obtain a section $\mathrm{Alt}(\jet^1\pi_*) \in \mathrm{Hom}(\wedge^2TN, \pi^*TM)$.

\begin{definition}
The \textbf{structure tensor} of $\pi: N \to M$ is the map
\[c: \Jet^1N \to \mathrm{Hom}(\wedge^2TM, \pi^*TM)\]
defined by
\[c(H_p)(v\wedge w) = \mathrm{Alt}(\jet^1\pi_*)(\tilde{v}\wedge \tilde{w}),\]
where $\tilde{v}$ and $\tilde{w}$ are the horizontal lifts of $v, w \in T_{\pi(p)}M$ to $H_p$.   
\end{definition} 

\begin{remark}
In the definition above, we have identified $\jet^1_x s \in \Jet^1N$ with the horizontal subspace $H_{s(x)} = \d_x s(T_xM)$ of $T_{s(x)}N$.
\end{remark}

\begin{example}
In the case where $N$ is the frame bundle  of $M$ (or more generally a $G$-structure over $M$), the structure tensor is simply its first order structure function (see, for example, \cite{Sternberg, SingerSternberg} or \cite{FernandesStruchiner}).
\end{example}

After these preliminary remarks, we now return to intrinsic construction of the classifying Lie algebroid of a fully regular coframe $\theta$ on a manifold $M$. Let us denote by $\B(M)$ the frame bundle of $M$. Thus,
\[\xymatrix{\B(M) \ar[d]_{\pi} & \ar@(dr,ur)@<-5ex>[]_{\GL_n}\\
M}
\]
is a principal $\GL_n$-bundle whose fiber over a point $p$ is
\[\pi^{-1}(p) = \set{\phi: \Rr^n \to T_pM: \phi \text{ is a linear isomorphism}}.\]

The coframe $\theta$ can be identified with a section of $\B(M)$. It gives rise to trivializations $TM \cong M \times \Rr^n$ and $T^*M \cong M \times \Rr^n$ and with respect to the canonical base of $\Rr^n$, we can write $\theta = (\theta^1, \ldots, \theta^n)$ or dually $\frac{\partial}{\partial \theta} = (\theta_1, \ldots, \theta_n)$. A basis of sections of $\wedge^2T^*M \otimes TM$ is then given by
\[\theta^{ij}_k =  (\theta^i \wedge\theta^j) \otimes \theta_k.\]
With respect to this basis, the structure tensor $c$ applied to $\jet^1\theta$ gives rise to a section of $\wedge^2T^*M \otimes TM$ whose coordinates are the structure functions, i.e.,
\[c(\jet^1\theta) = \sum C^k_{ij}\theta^{ij}_k.\]
It follows that we may view the set $\F_0(M)$ of structure functions of $\theta$ as the section $c(\jet^1\theta)$ of $\mathrm{Hom}(\wedge^2TM, TM)$.

Note also that $\jet^1 c$ is a map
\[\Jet^1(\Jet^1\B(M)) \to \Jet^1(\mathrm{Hom}(\wedge^2TM, \pi^*TM)).\]
Recall that $\Jet^1(\mathrm{Hom}(\wedge^2TM, \pi^*TM))$ is isomorphic to
\[(\wedge^2T^*M \otimes TM) \times \mathrm{Hom}(TM, \wedge^2T^*M \otimes TM).\]
Let us denote by $\nabla$ the flat connection on $\wedge^2T^*M \otimes TM$ induced be $\theta$. Then a basis of sections of $\mathrm{Hom}(TM, \wedge^2T^*M \otimes TM)$ is given by
\[\theta^{ij}_{k,l} = \nabla_{\theta_l}\theta^{ij}_k.\]
With respect to this basis, the section $\jet^1c(\jet^2\theta)$ of $\Jet^1( \mathrm{Hom}(\wedge^2TM, TM))$ is written as
\[\jet^1c(\jet^2\theta) = \sum C^k_{ij}\theta^{ij}_k + \frac{\partial C^k_{ij}}{\partial \theta_l}\theta^{ij}_{k,l}.\]
Thus we may view the set $\F_1(M)$ formed by the structure functions and its coframe derivatives as the section $\jet^1c(\jet^2\theta)$ of 
\[\Jet^1( \mathrm{Hom}(\wedge^2TM, TM)).\]
If we continue in this way, we may identify the set $\F_r(M)$ of all coframe derivatives of the the structure functions up to order $r$ with the section 
\[\jet^{r-1}c(\jet^r \theta) \in \Gamma(\Jet^r( \mathrm{Hom}(\wedge^2TM, TM))),\]
which we shall call the \textbf{$r$-th order structure section} of $\theta$.

Now, the coframe $\theta$ induces a trivialization of $\Jet^r( \mathrm{Hom}(\wedge^2TM, TM))$. Thus we have an identification between all its fibers. We shall denote them by $\mathbb{K}_r$ and by 
\[p^r_{\theta}: \Jet^r( \mathrm{Hom}(\wedge^2TM, TM)) \to \mathbb{K}_r\] 
the associated projection. We observe that the coframe $\theta$ is fully regular if and only if each of the maps
\[\xymatrix{ M  \ar@/_1.0cm/[rrrrrr]_{\tilde{\kappa}_r} \ar[rrr]^-{\jet^{r-1}c(\jet^r \theta)} &&&  \Gamma(\Jet^r( \mathrm{Hom}(\wedge^2TM, TM)))\ar[rrr]^-{p^r_{\theta}} &&& \mathbb{K}_r}\]
has constant rank.

Thus, we can summarize the construction made so far by saying that each coframe $\theta$ on $M$, viewed as a section of $\B(M)$, gives rise to a section 
\[ s_{\theta}\in \Gamma(\Jet^{\infty}(\mathrm{Hom}(\wedge^2TM,TM)).\] 
Moreover, since the coframe $\theta$ induces a trivialization of $\Jet^{\infty}(\mathrm{Hom}(\wedge^2TM,TM))$, there is a natural projection
\[p_{\theta}: \Jet^{\infty}(\mathrm{Hom}(\wedge^2TM,TM)) \to \mathbb{K}_{\infty},\]
where $\mathbb{K}_{\infty}$ denotes the fiber of $\Jet^{\infty}(\mathrm{Hom}(\wedge^2TM,TM))$. If, additionally, we assume that $\theta$ is fully regular, then $\tilde{\kappa}_{\theta} = p_{\theta} \circ s_{\theta}$ has constant rank, and thus its image in $\mathbb{K}_{\infty}$ is an immersed submanifold (possibly with self intersection) which we denote by $\tilde{X}$, i.e., there is a manifold $X$, and an immersion (not necessarily injective) $\phi: X \to \mathbb{K}_{\infty}$ such that
\[\xymatrix{ M \ar[drrr]_{\tilde{\kappa}_{\theta}} \ar@/^1.2cm/[rrrrrr]^{\tilde{\kappa}_\infty} \ar[rrr]^-{s_\theta} &&&  \Gamma(\Jet^{\infty}( \mathrm{Hom}(\wedge^2TM, TM)))\ar[rrr]^-{p_{\theta}} &&& \mathbb{K}_{\infty}\\
&&& X \ar[rrru]_{\phi}&&&}\]

We remark that again, as in Section \ref{subsec:coframes:2}, we may take $X$ to be the quotient of $M$ by the equivalence relation where two  points $p$ and $q$ are identified if and only if there is a locally defined formal equivalence of $\theta$ which takes $p$ to $q$. Thus, by construction, we may view the structure section of order zero $c(\jet^1\theta) \in \Gamma(\mathrm{Hom}(\wedge^2TM, TM))$ as being defined on $X$. In other words, there is a smooth map 
\[\hat{c} : X \to \mathrm{Hom}(\wedge^2TM, TM),\]
such that
\[\xymatrix{M \ar[dr]_-{\jet^1c(\theta)} \ar[rr]^{\tilde{\kappa}_{\theta}} && X \ar[dl]^{\hat{c}}\\
& \mathrm{Hom}(\wedge^2TM, TM) &}\]
It is convenient to use the coframe $\theta$ to fix a trivialization of $TM$ so that $\hat{c}$ becomes a map
\[\hat{c}: X \to \mathrm{Hom}(\wedge^2\Rr^n,\Rr^n).\]

Finally, if we define $n$ vector fields on $X$ through
\[F_i = (\kappa_{\theta})_*\frac{\partial}{\partial \theta_i} \in \X(X),\]
then we can describe the classifying Lie algebroid $A_{\theta} \to X$ of $\theta$ explicitly as follows.
We take $A_{\theta}$ to be the trivial vector bundle $A_{\theta} = X \times \Rr^n$ and define its structure by
\begin{align*}
[e_i, e_j]&= -\hat{c}(e_i\wedge e_j)\\
\sharp(e_i) & = F_i
\end{align*}
where $\{e_1,\dots,e_n\}$ denotes de canonical basis of sections.

\bibliographystyle{amsplain}
\bibliography{bibliography}

\providecommand{\bysame}{\leavevmode\hbox to3em{\hrulefill}\thinspace}
\providecommand{\MR}{\relax\ifhmode\unskip\space\fi MR }
\providecommand{\MRhref}[2]{%
  \href{http://www.ams.org/mathscinet-getitem?mr=#1}{#2}
}
\providecommand{\href}[2]{#2}
\begin{thebibliography}{10}

\bibitem{Bryant}
Robert~L. Bryant, \emph{Bochner-{K}{\"a}hler metrics}, J. of Amer. Math. Soc.
  \textbf{14} (2001), no.~3, 623--715.

\bibitem{CannasWeinstein}
Ana Cannas~da Silva and Alan Weinstein, \emph{Geometric models for
  noncommutative algebras}, Berkeley Mathematics Lecture Notes, vol.~10,
  American Mathematical Society, Providence, RI, 1999. \MR{MR1747916
  (2001m:58013)}

\bibitem{Cartan}
\'{E}. Cartan, \emph{Sur la structure des groupes infinis de transformations},
  Ann. \'{E}c. Norm. \textbf{3} (1904), 153--206.

\bibitem{Crainic:Van}
M.~Crainic, \emph{Differentiable and algebroid cohomology, {V}an {E}st
  isomorphism, and characteristic classes}, Comment. Math. Helv. \textbf{78}
  (2003), 681--721.

\bibitem{CrainicFernandes:exotic}
M.~Crainic and R.~L. Fernandes, \emph{Secondary characteristic classes of {L}ie
  algebroids}, Quantum field theory and noncommutative geometry, Lecture Notes
  in Phys., vol. 662, Springer, Berlin, 2005, pp.~157--176. \MR{MR2179182
  (2007b:53047)}

\bibitem{CrainicFernandes:lecture}
\bysame, \emph{Lectures on integrability of {L}ie brackets}, Geometry \&
  Topology Monographs \textbf{17} (2010), 1--94.

\bibitem{CrainicFernandes}
Marius Crainic and Rui~Loja Fernandes, \emph{On integrability of {L}ie
  brackets}, Annals of Mathematics \textbf{157} (2003), 575--620.

\bibitem{Weinstein:modular1}
S.~Evans, J.~H Lu, and A.~Weinstein, \emph{Transverse measures, the modular
  class and a cohomology pairing for {L}ie algebroids}, Quart. J. Math.
  \textbf{50} (1999), no.~Ser. 2, 417--436.

\bibitem{FernandesStruchiner}
R.~L. Fernandes and I.~Struchiner, \emph{The classifying {L}ie algebroid of a
  geometric structure {II}: {$G$}-structures and examples}, preprint in
  preperation.

\bibitem{Fernandes:holonomy}
Rui~Loja Fernandes, \emph{Lie algebroids, holonomy and characteristic classes},
  Adv. Math. \textbf{170} (2002), no.~1, 119--179. \MR{MR1929305 (2004b:58023)}

\bibitem{Kobayashi}
S.~Kobayashi, \emph{Transformation groups in differential geometry},
  Springer-Verlag, 1972.

\bibitem{Weinstein:modular2}
Y~Kosmann-Schwarzbach, C.~Laurent-Gengoux, and A.~Weinstein, \emph{Modular
  class of {L}ie algebroid morphisms}, Preprint availible at arXiv:0712.3021
  [math.DG] (2008), 1--33.

\bibitem{Olver}
Peter~J. Olver, \emph{Equivalence, invariants and symmetry}, Cambridge
  University Press, 1995.

\bibitem{Olver:Lie}
\bysame, \emph{Non-associative local {L}ie groups}, J. Lie Theory \textbf{6}
  (1996), 23--51.

\bibitem{Palais}
R.~S. Palais, \emph{A global formulation of the {L}ie theory of transformation
  groups}, Mem. Amer. Math. Soc \textbf{22} (1957), iii+123pp.

\bibitem{Schwachhofer:integration}
Proceedings of the "Rencontres Math\'ematiques de Glanon", \emph{Symplectic
  connections via integration of {P}oisson structures}, 2002.

\bibitem{Sharpe}
R.~W. Sharpe, \emph{Differential geometry: Cartan's generalization of klein's
  erlangen program}, Graduate Texts in Mathematics, vol. Volume 166, Springer,
  1997.

\bibitem{SingerSternberg}
I.~M. Singer and Shlomo Sternberg, \emph{The infinite groups of {L}ie and
  {C}artan}, J. D'Anal. Math. \textbf{15} (1965), 1--113.

\bibitem{Sternberg}
Shlomo Sternberg, \emph{Lectures on differential geometry}, Prentice-Hall,
  1964.

\bibitem{Struchiner}
I.~Struchiner, \emph{The classifying lie algebroid of a geometric structure},
  Ph.D. thesis, University of Campinas - UNICAMP, 2009.

\end{thebibliography}

\end{document}